\documentclass[11pt, oneside]{amsart}
\usepackage{geometry}                % See geometry.pdf to learn the layout options. There are lots.
\usepackage{graphicx}
\usepackage{amssymb}
\usepackage{epstopdf}
\usepackage{mathtools}
\usepackage{blkarray}
\usepackage{tikz-cd}
\usepackage{amsmath}
\newtheorem{thm}{Theorem}[section]
\newtheorem{cor}[thm]{Corollary}
\newtheorem{prop}[thm]{Proposition}
\newtheorem{lem}[thm]{Lemma}

\theoremstyle{remark}
\newtheorem{rem}[thm]{Remark}
\theoremstyle{definition}
\newtheorem{defn}[thm]{Definition}
\newtheorem{ex}[thm]{Example}

\DeclareMathOperator{\Conf}{Conf}

\DeclareMathOperator{\Top}{Top}
\DeclareMathOperator{\Sp}{Sp}
\DeclareMathOperator{\Ab}{Ab}
\DeclareMathOperator{\Mod}{Mod}

\DeclareMathOperator{\Ch}{Ch}

\DeclareMathOperator{\Hom}{Hom}

\DeclareMathOperator{\colim}{colim}
\DeclareMathOperator{\hocolim}{hocolim}
\DeclareMathOperator{\FI}{FI}
\DeclareMathOperator{\GI}{GI}

\newcommand{\tinytriangle}{{\includegraphics[height=1.5ex]{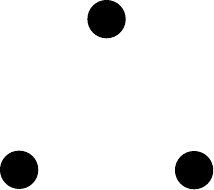}}}
\newcommand{\tinytriangleab}{{\includegraphics[height=1.5ex]{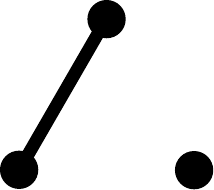}}}
\newcommand{\tinytriangleac}{{\includegraphics[height=1.5ex]{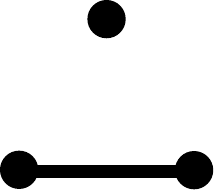}}}
\newcommand{\tinytrianglebc}{{\includegraphics[height=1.5ex]{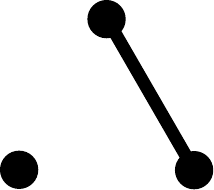}}}
\newcommand{\tinytriangleabac}{{\includegraphics[height=1.5ex]{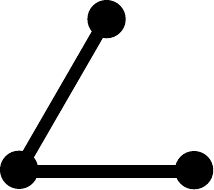}}}
\newcommand{\tinytriangleabbc}{{\includegraphics[height=1.5ex]{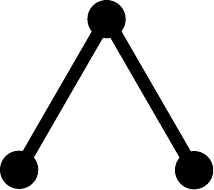}}}
\newcommand{\tinytriangleacbc}{{\includegraphics[height=1.5ex]{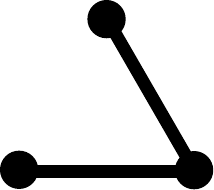}}}
\newcommand{\tinytriangleabacbc}{{\includegraphics[height=1.5ex]{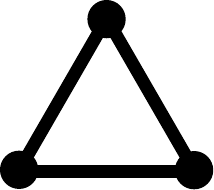}}}

\newcommand{\tinyintervalab}{{\includegraphics[height=1.5ex]{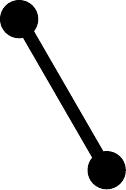}}}
\newcommand{\tinyinterval}{{\includegraphics[height=1.5ex]{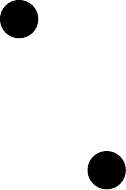}}}

\title{Configuration space in a product}
\author{John D. Wiltshire-Gordon}
\begin{document}
\maketitle
\begin{abstract}
Given a finite graph $\Gamma$ and a topological space $Z$, the graphical configuration space $\Conf(\Gamma, Z)$ is the space of functions $\mathcal{V}(\Gamma) \to Z$ so that adjacent vertices map to distinct points.  We provide a homotopy decomposition of $\Conf(\Gamma, X \times Y)$ in terms of the graphical configuration spaces in $X$ and $Y$ individually.  By way of application, we prove a stabilization result for homology of configuration space in $X \times \mathbb{C}^p$ as $p \to \infty$.  We also compute $H_\bullet \Conf(K_3,T)/T$, the integral homology of the space of ordered triples of distinct points in a torus $T=\mathbb{R}^r/\mathbb{Z}^r$ of rank $r$, where configurations are considered up to translation.  In \S\ref{sec:method}, we give an algorithm for computing homology of configuration space in a product of simplicial complexes.  The method is applied to products of some sans-serif capital letters in Example~\ref{ex:computation}.
\end{abstract}
\section{Introduction}
\noindent
Define the configuration space of $n$ distinct, labeled points in a topological space $Z$
$$
\Conf(n, Z) = \{(z_1, \ldots, z_n) \mbox{ such that $z_i = z_j \implies i=j$ } \}.
$$
Our aim is to understand the space $\Conf(n, Z)$ if $Z=X \times Y$ is a product of two spaces.  As a consequence of our analysis, we obtain the following homological stability result:
\begin{thm} \label{thm:stable}
If $X$ is a finite simplicial complex, and if $m, b \in \mathbb{N}$ are natural numbers,
$$
H_{mp + b} \Conf(n, X \times \mathbb{C}^p)
$$
stabilizes to a single abelian group as $p \to \infty$, and this group is $0$ if $m$ is odd.
\end{thm}
\noindent
For example, writing $\textsf{Y}$ for the cone on three points,
\begin{align*}
H_{2p + 1} \Conf(3, \textsf{Y} \times \mathbb{C}^p) &= \mathbb{Z}^3 \\
H_{4p + 1} \Conf(3, \textsf{Y} \times \mathbb{C}^p) &= \mathbb{Z}^{10},
\end{align*}
for all $p > 2$.  In Theorem \ref{thm:stableformula}, we provide a bound for stabilization and a formula for the limiting groups.

\begin{rem}
Theorem \ref{thm:stable} is already known if $X$ is a manifold by \cite[Remark 2.3]{CT78}. 
\end{rem}

We give one more application before attending to the main results.  
Let $T = \mathbb{R}^r/\mathbb{Z}^r$ be a real torus of rank $r \geq 2$, and write $\Conf(3, T)/T$ for the configuration space of ordered triples in $T$ considered up to simultaneous translation.
\begin{thm} \label{thm:torus-mod-rotation}
The groups $H_p (\Conf(3, T)/T)$ are torsion-free for all $p$.  The top non-vanishing 
Betti number is $\beta_{2r-2} =  r \cdot (r + 3)/2$, 
and in all lower degrees the Betti numbers are given by the formula
$$
\beta_p = \binom{2 r}{p}-3 \binom{r}{p-r}.
$$
\end{thm}

\begin{rem}
For even $r$, the space $\Conf(3, T)/T$ has appeared in the study of abelian arrangements.  Specifically, if $A$ is a complex abelian variety of real dimension $r$, then $\Conf(3, T)/T$ is homeomorphic to the arrangement complement
$$
\{ \mbox{ $(a_1, a_2) \in A^2$ so that $a_1 \neq 0$ and $a_2 \neq 0$ and $a_1 \neq a_2$ } \}.
$$
The rational cohomology of this space may be found using a theorem of Dupont, which provides a rational model for certain hypersurface complements \cite{Dupont15}.  The rational cohomology, and its weight filtration, may also be computed from a theorem of Bibby; see \cite[Example 4.2]{Bibby16}.
\end{rem}

\noindent
For ease of exposition, we begin our discussion of the main results with the case $n=2$.
\subsection{Ordered pairs of distinct points} \label{sec:pairs}
The configuration space of ordered pairs of distinct points in $Z$,
$$
\Conf(2, Z) = Z \times Z - \{(z,z) \mbox{ for } z \in Z \},
$$
is also known as the \textbf{deleted diagonal}.  
Suppose that $Z = X \times Y$ factors as a product of two spaces.  We wish to understand $\Conf(2,Z)$ in terms of $X$ and $Y$.

If $z_1 = (x_1, y_1)$ and $z_2 = (x_2, y_2)$ are two points of $Z$, then there are three ways we might have $(z_1, z_2) \in \Conf(2, Z)$:  $(x_1, x_2) \in \Conf(2, X)$, or $(y_1, y_2) \in \Conf(2, Y)$, or both.  In other words, we have a pushout diagram
\begin{equation} \label{eq:basic}
\begin{tikzcd}
\Conf(2, X) \times \Conf(2, Y) \ar[r] \arrow{d} & X^2 \times \Conf(2, Y) \ar[d] \\
\Conf(2, X) \times Y^2 \arrow{r} & \Conf(2, X \times Y).
\end{tikzcd}
\end{equation}
From (\ref{eq:basic}) we see that $X^2$ and $Y^2$ want to be treated as a configuration spaces as well.   Introduce \textbf{graphical configuration space}.  If $\Gamma$ is a graph with vertices $\{1, \ldots, n\}$, let
$$
\Conf(\Gamma, X) = \{ (x_1, \ldots, x_n) \in X^n \mbox{ so that $i \sim_{\Gamma} j \implies x_i \neq x_j$ } \},
$$
where $i \sim_{\Gamma} j$ indicates the existence of an edge in $\Gamma$ connecting $i$ and $j$.
Rewriting (\ref{eq:basic}),
\begin{equation} \label{eq:basic2}
\begin{tikzcd}
\Conf(\, \tinyintervalab \;, X) \times \Conf(\, \tinyintervalab \;, Y) \ar[r] \arrow{d}  & \Conf(\, \tinyinterval \;, X) \times \Conf(\, \tinyintervalab \;, Y) \arrow{d}{\beta} \\
\Conf(\, \tinyintervalab \;, X) \times \Conf(\, \tinyinterval \;, Y) \arrow{r}{\alpha} & \Conf(2, X \times Y).
\end{tikzcd}
\end{equation}
If the spaces $X$ and $Y$ are Hausdorff, then the maps $\alpha$ and $\beta$ are open immersions.  Defining $U = \mathrm{im}(\alpha)$, $V = \mathrm{im}(\beta)$, we find that $\{U, V\}$ is an open cover of $\Conf(2, X \times Y)$ satisfying
\begin{align*} 
U &\cong  \Conf(\, \tinyintervalab \;, X) \times \Conf(\, \tinyinterval \;, Y) \\
V &\cong \Conf(\, \tinyinterval \;, X) \times \Conf(\, \tinyintervalab \;, Y) \\
U \cap V &\cong \Conf(\, \tinyintervalab \;, X) \times \Conf(\, \tinyintervalab \;, Y).
\end{align*}
The Mayer-Vietoris sequence in homology
$$
\cdots \to H_i (U \cap V) \to H_i U \oplus H_i V \to H_i \Conf(2, X \times Y) \xrightarrow{\partial_i} H_{i-1} (U \cap V) \to \cdots,
$$
tightly constrains the homology of $\Conf(2, X \times Y)$ in terms of the homologies of $U$, $V$, and $U \cap V$.  Our next steps provide a method of computing the connecting maps $\partial_i$.

A general pushout of spaces is homotopically ill-behaved.  However, the pushout of (\ref{eq:basic2}) is really the union of two open sets, and such pushouts are weakly equivalent to homotopy pushouts.  Specifically,
$$
\Conf(2, X \times Y) \simeq (U \times [0, 1]) \sqcup (V \times [0, 1]) \Big / \sim
$$
where $\sim$ identifies the two copies of $(U \cap V) \times 0$.  Due to its homotopical nature, this description depends only on the homotopy types of the space-subspace pairs
$$
(\Conf(\, \tinyinterval \;, X) \,, \Conf(\, \tinyintervalab \;, X)) \;\;\;\mbox{ and } \;\;\; (\Conf(\, \tinyinterval \;, Y) \,, \Conf(\, \tinyintervalab \;, Y)).
$$
We have therefore succeeded in providing a description $\Conf(2, X \times Y)$ in terms of $X$ and $Y$ separately.  Moreover, the homology of $\Conf(2, X \times Y)$ may be computed by a double complex spectral sequence whose $E^2$ page coincides with the Mayer-Vietoris sequence.  This double complex enables the calculation of $H_i \Conf(2, X \times Y)$ as well as the connecting maps.

\subsection{An open cover for graphical configuration space in a product} \label{sec:cover}
The results of this paper ultimately rest on three elementary lemmas that describe an organized open cover of configuration space, generalizing the cover found in \S\ref{sec:pairs}.

Given graphs $\Gamma'$ and $\Gamma''$ on the nodes $\{1, \ldots, n\}$, define a subset $\mathcal{U}_{\Gamma', \, \Gamma''} \subseteq (X \times Y)^n$
$$
\mathcal{U}_{\Gamma', \, \Gamma''} = \{ ((x_1, y_1), \ldots, (x_n, y_n)) \mbox{ so that $i \sim_{\Gamma'} j \implies x_i \neq x_j$ and $i \sim_{\Gamma''} j \implies y_i \neq y_j$ } \}.
$$
Under our standing assumption that $X$ is Hausdorff, this subset is guaranteed to be open.  Moreover, if the spaces $X \in \Top_G$ and $Y \in \Top_H$ carry actions of topological groups $G$ and $H$, then the open subset $\mathcal{U}_{\Gamma', \, \Gamma''}$ is stable under the action of $G \times H$.

In the new notation, the open cover from  \S\ref{sec:pairs} reads
$$
\Conf(2, X \times Y) = \mathcal{U}_{\;\tinyintervalab \;,\; \tinyinterval} \; \cup \; \mathcal{U}_{\;\tinyinterval \;,\; \tinyintervalab}
$$

\begin{lem} \label{lem:split}
For all $\Gamma', \Gamma''$, we have an evident homeomorphism
$$
\mathcal{U}_{\Gamma', \, \Gamma''} \cong \Conf(\Gamma', X) \times \Conf(\Gamma'', Y).
$$
\end{lem}

\begin{lem} \label{lem:intersections} 
The intersection of two such opens is given by taking unions of graphs
$$
(\mathcal{U}_{\Gamma_1', \,\Gamma_1''}) \cap (\mathcal{U}_{\Gamma_2', \, \Gamma_2''}) = \mathcal{U}_{(\Gamma_1' \cup \Gamma_2'), \,(\Gamma_1'' \cup \Gamma_2'')}.
$$
\end{lem}

\begin{lem} \label{lem:cover}
For all graphs $\Gamma$ on the nodes $\{1, \ldots, n\}$, we have
$$
\Conf(\Gamma, X \times Y) = \bigcup_{\Gamma', \, \Gamma''} \mathcal{U}_{\Gamma', \, \Gamma''},
$$
where the union is over all pairs of subgraphs $\Gamma', \Gamma''$ with $\Gamma' \cup \Gamma'' = \Gamma$.
\end{lem}
\begin{proof}
Two ordered pairs $(x_i, y_i)$ and $(x_j, y_j)$ are distinct if and only if
$$
x_i \neq x_j \;\;\;\mbox{ or } \;\;\;y_i \neq y_j. \qedhere
$$
\end{proof}
\noindent

\subsection{Main results}
We obtain suitable generalizations of the ideas and results already indicated for $n=2$ in \S\ref{sec:pairs}.  Every aspect of the argument upgrades in a natural way, corresponding to the combinatorics explained in \S\ref{sec:cover}.

Write $\mathcal{G}(n)$ for the poset of graphs on the vertices $\{1, \ldots, n\}$ ordered by graph inclusion.  The coproduct in this poset, written
$$
U_n \colon \mathcal{G}(n) \times \mathcal{G}(n) \to \mathcal{G}(n),
$$
sends a pair of graphs to their union.  Any $\mathcal{G}(n)$-morphism $\Gamma' \subseteq \Gamma$ induces an inclusion 
$$
\Conf(\Gamma, X) \subseteq \Conf(\Gamma', X)
$$
in the opposite direction.  This construction makes $\Conf(-,X)$ into a functor
$$
\Conf(-,X) \colon \mathcal{G}(n)^{op} \to \Top.
$$
Let $G,H$ be topological groups, and let $X \in \Top_G$ and $Y \in \Top_H$ be Hausdorff spaces with group actions.
\begin{thm} \label{thm:pointwise}
For every $n \in \mathbb{N}$ there is a weak equivalence of $\mathcal{G}(n)^{op}$-shaped diagrams
$$
\mathbb{L}(U_n^{op})_! \left[ \Conf(-, X) \times \Conf(-, Y) \right] \overset{\sim}{\longrightarrow} \Conf(-,X \times Y)
$$
where $\mathbb{L}(U_n^{op})_!$ denotes homotopy left Kan extension along the functor $U_n^{op}$.
\end{thm}

\begin{cor}
The homotopy type of $\Conf(n, X \times Y)$ depends only on the homotopy types of the $\mathcal{G}(n)^{op}$-shaped diagrams $\Gamma' \mapsto \Conf(\Gamma', X)$ and $\Gamma'' \mapsto \Conf(\Gamma'', Y)$.
\end{cor}
\begin{proof}
Homotopy invariance of products and homotopy left Kan extensions.
\end{proof}

We now state a result that applies to all $n$ at once.  Write $\GI$ for the category of finite graphs with injections, and $U \colon \GI_2 \to \GI$ for the functor that unions a pair of graphs with the same underlying vertex set.  The details of this notation are given in \S\ref{sec:notation} and \S\ref{sec:action}.

\begin{thm} \label{thm:main}
There is a weak equivalence
$$
\mathbb{L}(U^{op})_! \left[ \Conf(-, X) \times \Conf(-, Y) \right] \overset{\sim}{\longrightarrow} \Conf(-, X \times Y)
$$
in the category of functors $\GI^{op} \to \Top_{G \times H}$.
\end{thm}

\begin{rem}
Theorems \ref{thm:pointwise} and \ref{thm:main} may be iterated, providing a homotopy decomposition for configuration space in a product of several factors.  The proof of Theorem \ref{thm:torus-mod-rotation} relies heavily on this idea.
\end{rem}

\subsection{Relation to prior work}
We tabulate several results describing various groups $H_\bullet( \Conf(n, Z)\, ; \, R)$, and take note of their assumptions and the explicitness of their answers.
\\ \\
\begin{tabular}{cccc}
\hline
 $Z$ & $R$ & style of description & citation \\ \hline
 $\mathbb{C}$ & $\mathbb{C}$ & full answer & \cite{Arnold69}  \\
 $\mathbb{R}^k$ & $\mathbb{Z}$ & full answer & \cite{CohenThesis} \\
 $S^k$ & $\mathbb{Z}$ & full answer &  \cite{FZ00} \\
 $\mathbb{R}P^k$& $\mathbb{Z}[1/2]$ & full answer & \cite{GGX15} \\
 $M$ a manifold & field & spectral sequence & \cite{CT78} \\
 $M \times \mathbb{R}$ & field & explicit spectral sequence & \cite{CT78} \\
 smooth projective variety/$\mathbb{C}$ & $\mathbb{Q}$ & explicit spectral sequence &  \cite{Totaro96} \cite{Kriz94} \\
 simplicial complex & $ \mathbb{Z} $ & explicit chain complex & \cite{WG17}.\\ \hline
\end{tabular}
\\\vspace{.1in} \\  \noindent
Although the last three entries of the table are listed ``explicit,'' they are often computationally expensive, leaving lots of room between ``explicit'' and ``full answer''.  For example, combinatorial study of the spectral sequence from \cite{CT78} and \cite{Totaro96} can be quite involved; see \cite{Maguire16} for the case of $\mathbb{C}P^3$.  The complex found in \cite{WG17} is worse still, usually much too large to be useful in computation if $\dim Z >2$.  

This paper does not fit snugly in the table, since it is of a different style.  We produce information about more-complicated configuration spaces by looking at simpler ones.  For example, our proof of Theorem \ref{thm:torus-mod-rotation} relies on a simple---but \textit{ad hoc}---model of
$$
C_\bullet^{sing} \Conf(-,\mathbb{R}/\mathbb{Z}) \colon \mathcal{G}(3)^{op} \to \Ch_\bullet(\mathbb{Z})
$$
that we construct by hand.  If the more-powerful results in the table could be adapted to give systematic chain models for graphical configuration spaces, then Theorem \ref{thm:pointwise} could be employed to compute many new groups.

\begin{rem}
The techniques of \cite{WG17} do adapt to this setting, as we demonstrate in \S\ref{sec:method}, where we use Theorem \ref{thm:pointwise} to compute the homology of a few small configuration spaces.  See Remark \ref{rem:wg17} for more details.
\end{rem}

\begin{rem}
The recent paper ``Configuration spaces of products'' \cite{DHK17}  has some philosophical similarity to this one.  They also obtain a homotopical description of configuration space in terms of a derived tensor product, in their case, a tensor product of modules over an operad.  As has been typical in the study of configuration space, their work applies only to manifolds\footnote{We must mention, however, the study of configuration space in topological graphs, including connections to representation stability; see \cite{Ramos2018} and \cite{An:2018aa}, for example, and their bibliographies.}.  We also mention the paper ``The configuration category of a product'' \cite{dBW17} which provides a similar description of a related configuration space, still with the manifold assumption.
\end{rem}

\begin{rem}[Representation stability] \label{rem:repstability}
According to \cite{ChurchEllenbergFarb15}, the $\FI^{op}$ structure on configuration space gives a finitely generated $\FI$-action on cohomology if $Z$ is an oriented, finite-type manifold of dimension at least two.  Tosteson proves a similar result that also applies to certain singular spaces \cite{Tosteson16}.

In principle, if enough values for small $n$ can be computed, then representation stability provides the rest.  Unfortunately, this approach is not currently feasible because the ``small $n$'' computations are still too difficult. 
\end{rem}
\begin{rem}
In \S\ref{sec:method}, we use Theorem \ref{thm:pointwise} in conjunction with the derived category of $\mathcal{G}(n)^{op}$-modules to make homology calculations one $n$ at a time.  However, in light of Remark \ref{rem:repstability}, it would be more natural and forceful to use Theorem \ref{thm:main} in conjunction with the derived category of $\GI$-modules to compute with all $n$ at once.  There are two difficulties that arise.  

First, cofibrant chain models for graphical configuration space will necessarily be quite complicated, even for configurations in simple spaces.  Second, there is no technology in place for the derived category of $\GI$-modules.  Work in this direction could make use of a growing understanding of the derived category of $\FI$-modules; see \cite[Theorem 2.8]{NS17} for example.
\end{rem}

\subsection{Notation for graphs and injections} \label{sec:notation}
By a finite graph $\Gamma = (V, \sim)$, we mean a finite set $V$ equipped with a relation $\sim$ satisfying symmetry and total non-reflexivity:
$$
v \sim u \; \Longleftrightarrow u \sim v \;\; \mbox{ and } \;\; v \not \sim v \;\;\;\;\;\;\;\;\;\mbox{for all $u, v \in V$.}
$$
  We shall call a relation satisfying these two conditions a \textbf{graphical relation}.

Given two graphs $\Gamma_1 = (V_1, \sim_1)$ and $\Gamma_2 = (V_2, \sim_2)$, a \textbf{graph injection} $\Gamma_1 \to \Gamma_2$ is a set injection $\varphi \colon V_1 \hookrightarrow V_2$ preserving the graphical relation:
$$
u \sim_1 v \; \implies \varphi(u) \sim_2 \varphi(v).
$$
The category of finite graphs with graph injections is denoted $\GI$.  We also need a category called $\GI_2$ whose objects are triples $(V;\; \sim',\, \sim'')$ of a finite set and two graphical relations, and where a morphism is required to preserve each graphical relation separately.  

Write $U$ for the \textbf{union-the-edges functor}, 
$$
\begin{array}{rcccl}
U & \colon & \GI_2 &\to& \GI \\
&& (V;\; \sim',\, \sim'') &\mapsto& (V, \; \sim' \cup \sim'' \;),
\end{array}
$$
 where $\sim' \cup \sim''$ denotes the union of relations considered as subsets of $V \times V$.

If $\Gamma = (V, \sim)$ is a graph, an \textbf{edge-subgraph} $\Gamma' \subseteq_e \Gamma$ is any graph $\Gamma' = (V, \sim')$ on the same vertices so that $u \sim' v \implies u \sim v$ for all $u, v \in V$.  Let $\mathcal{G}(\Gamma)$ be the poset of all edge-subgraphs of $\Gamma$, and write $\mathcal{P}(\Gamma) \subseteq \mathcal{G}(\Gamma) \times \mathcal{G}(\Gamma)$ for the subposet whose objects are pairs of edge-subgraphs $(\Gamma', \Gamma'')$ with $U(\Gamma', \Gamma'') = \Gamma$.

\subsection{Action of $\GI^{op}$ on graphical configuration space} \label{sec:action}
For any space $X$, a graph injection $\varphi \colon \Gamma_1 \to \Gamma_2$ induces a continuous map pointing the other way
$$
\Conf(\varphi, X) \colon \Conf(\Gamma_2, X) \to \Conf(\Gamma_1, X)
$$
given by forgetting points and relabeling.  Explicitly, if 
$$
\Gamma_1 = (\{1, \ldots\hspace{1.5pt}, \hspace{1.5pt} n\}, \sim_1) \mbox{      and      } \Gamma_2 = (\{1, \ldots , m\}, \sim_2),
$$
then the induced map on configuration space is given by the formula
$$
(x_1, \ldots, x_m) \mapsto (x_{\varphi(1)}, \ldots, x_{\varphi(n)}).
$$

In applications, the space $X$ may carry the action of a topological group $G$, and we accommodate this with the general definition
$$
\begin{array}{rcccl}
\Conf(-, X) & \colon & \GI^{op} &\to& \Top_G \\
&& \Gamma &\mapsto& \Conf( \Gamma , X),
\end{array}
$$
where $G$ acts on $\Conf(\Gamma, X)$ diagonally.  We now clarify some notation appearing in the statement of Theorem \ref{thm:main}.  If $X \in \Top_G$ and $Y \in \Top_H$ define the functor
$$
\begin{array}{rcccl}
\Conf(-, X) \times \Conf(-, Y) &\colon & \GI_2 &\to& \Top_{G \times H}\\
&&(V;\; \sim',\, \sim'')& \mapsto &\Conf(\Gamma', X) \times \Conf(\Gamma'', Y),
\end{array}
$$
where $\Gamma' = (V, \sim')$ and $\Gamma'' = (V, \sim'')$.

\subsection{Acknowledgements}
Thanks to Daniel Cohen, Jordan Ellenberg, Nir Gadish, Ben Knudsen, Aaron Mazel-Gee, Peter Patzt, Sam Payne, Eric Ramos, Daniel Ramras, and Jesse Wolfson for helpful conversations, and to Graham Denham, Giovanni Gaiffi, Rita Jim\'enez Rolland, and Alexander Suciu for organizing the Oberwolfach meeting at which the main results of this paper were announced \cite{MFO18}.  This work was inspired in part by the excellent paper \cite{MPY17}.  I learned how to prune a chain complex from the source code of Macaulay2 \cite{M2}.  This work was supported by the Algebra RTG at the University of Wisconsin, DMS-1502553.

\section{Computing homology using Theorem \ref{thm:pointwise}} \label{sec:method}
\noindent
After taking singular chains, the homotopy left Kan extension used in Theorem \ref{thm:pointwise} becomes a (purely algebraic) left derived functor.  We explain how to compute this functor by finding cofibrant chain models.

\subsection{Singular chains}
Given a space $X \in \Top_G$, we may compute its equivariant homology using the complex of singular chains on the Borel construction $C_\bullet^{sing}((X \times EG)/G)$.  This construction may be broken into three steps:
$$
\Top_G \xrightarrow{\hocolim_G} \Top \xrightarrow{\Sigma^\infty_+} \Sp \xrightarrow{- \wedge H\mathbb{Z}} \Mod_{H\mathbb{Z}}.
$$
The first map takes homotopy $G$-orbits; it is an $(\infty,1)$-analog of the usual quotient by $G$.  We write $X_{hG} = (X \times EG)/G$ for the image of $X$ under this map.

The second map sends $X_{hG}$ to its suspension spectrum $\Sigma^{\infty} (X_{hG})_+$.  This operation is stabilization in the sense of stable homotopy theory.  By the Freudenthal suspension theorem, we retain here all information about $X$ that is relevant to its homology.

The third map is smashing with $H\mathbb{Z}$, the Eilenberg-Mac Lane spectrum associated to the integers.  This operation is the stable homotopy version of ``taking homology,''  since the homotopy groups of the resulting space give the usual $G$-equivariant homology of $X$.

Since spaces have no integral homology in negative degree, the composite of these three maps actually lands in the subcategory of connective $H\mathbb{Z}$-modules, a category that is modeled by nonnegatively graded chain complexes of abelian groups via the usual Moore complex.

The following consequence is well-known to experts.

\begin{lem} \label{lem:commutes}
If $f \colon \mathcal{C} \to \mathcal{D}$ is a functor and $Z \colon \mathcal{C} \to \Top_G$ is a diagram, then there is a quasi-isomorphism of chain complexes
$$
C_{\bullet}^{sing} \mathbb{L}f_! Z \underset{qi}{\simeq} \mathbb{L}f_! C_{\bullet}^{sing} Z 
$$
where the first homotopy left Kan extension is taken in the category of $G$-spaces, and the second is taken in the category of chain complexes.
\end{lem}
\begin{proof}
Each of the three operations $\hocolim_G$, $\Sigma^{\infty}_+$, and $- \wedge H \mathbb{Z}$ are left $(\infty,1)$-adjoints, so they commute with homotopy left Kan extension.  
\end{proof}

\subsection{The derived category of $\mathcal{G}(n)^{op}$-modules}
In order to compute the derived functor from Lemma \ref{lem:commutes}, we will use the projective model structure on diagrams of chain complexes.

Let $\mathcal{A} = [\mathcal{G}(n)^{op}, \Ab]$ be the abelian category of presheaves of abelian groups on the poset $\mathcal{G}(n)$.  The computation we wish to perform occurs in the derived category $D_+(\mathcal{A})$ of chain complexes in nonnegative degree, and relies heavily on projective resolutions, so we begin with some basic facts about projective objects in $\mathcal{A}$.

There is a basic projective for every $\Gamma \in \mathcal{G}(n)$, given by the linearized representable functor
$$
\begin{array}{rcccl}
(-,\Gamma) & \colon & \mathcal{G}(n)^{op} & \to & \Ab \\
& & \gamma & \mapsto & \mathbb{Z} \cdot \Hom_{\mathcal{G}(n)}(\gamma, \Gamma).
\end{array}
$$
We use a similar notation for the linearized representable presheaves on any category.
\begin{lem}[Yoneda] \label{lem:yoneda}
If $\mathcal{C}$ is any category, and $(-,c)$ stands for the linearized presheaf represented by $c \in \mathcal{C}$, then for any presheaf $A \colon \mathcal{C}^{op} \to \Ab$,
$$
\Hom((-,c), A) \simeq Ac.
$$
As a consequence, $(-,c)$ is projective, since evaluation at $c$ is an exact functor.  Moreover, since any nonzero presheaf $A$ must have $Ac \neq 0$ for some $c \in \mathcal{C}$, the collection $\{\, (-,c) \,\}_{c \in \mathcal{C}}$ is a family of enough projectives.
\end{lem}
\noindent
From Lemma \ref{lem:yoneda}, we compute
$$
\Hom((-,\Gamma), (-,\Gamma')) = \begin{cases}
 \mathbb{Z} & \mbox{ if $\Gamma \subseteq \Gamma'$} \\
 0 & \mbox{ otherwise. }
\end{cases}
$$
It follows that a map from one direct sum of representable presheaves to another may be given by an integer matrix with certain forced zeros.  For example, and to solidify conventions, set $n=3$ and consider the map of presheaves
$$
(-, \tinytriangleab) \oplus (-, \tinytriangleabacbc) \xrightarrow{\left[\begin{array}{ccc}
 2 & 0 & 3 \\
 0 & 0 & 5 \\
\end{array}\right]} (-, \tinytriangleabac) \oplus (-, \tinytriangleac) \oplus (-, \tinytriangleabacbc).
$$
The three zeros appearing in this matrix are forced.  In abbreviated form, this map reads
$$
\begin{blockarray}{cccc} & \tinytriangleabac & \tinytriangleac & \tinytriangleabacbc \\
 \begin{block}{c[ccc]}
 \tinytriangleab & 2 & 0 & 3 \\
 \tinytriangleabacbc & 0 & 0 & 5 \\
 \end{block}
 \end{blockarray}.
$$
Our chain complexes are concentrated in nonnegative degree, and use homological grading conventions so that the differential decreases the degree by one.  A typical complex $C_\bullet$ looks like
$$
\cdots \longrightarrow \bullet_3 \overset{\partial_2}{\longrightarrow} \bullet_2 \overset{\partial_1}{\longrightarrow} \bullet_1 \overset{\partial_0}{\longrightarrow} \bullet_0,
$$
where the differentials $\partial_i$ are usually written out with explicit matrices.  We display only the nonzero degrees and differentials.
We write $C_\bullet[1]$ for the shifted complex
$$
\cdots \overset{\partial_2}{\longrightarrow} \bullet_3 \overset{\partial_1}{\longrightarrow} \bullet_2 \overset{\partial_0}{\longrightarrow} \bullet_1,% \overset{0}{\longrightarrow} \bullet_0.
$$
and similarly $C_{\bullet}[k]$ for higher shifts.  If $A \in \mathcal{A}$ is a presheaf, we write $A[k]$ for $A$ considered as a complex concentrated in degree $k$.  Similarly, if $\Gamma \in \mathcal{G}(n)$ is a graph, we write $\Gamma[k]$ for $(-, \Gamma)[k]$.

In the projective model structure on $D_+(\mathcal{A})$, a chain complex is cofibrant if it is projective in every degree.  For example, if $V \in \mathcal{A}$ is a presheaf, then a cofibrant model for $V[0]$ is the same as a projective resolution of $V$.

Our conventions are chosen to make chain complexes typographically compact.  For example, if $V$ is the skyscraper presheaf that takes the value $\mathbb{Z}$ on $\tinytriangleabacbc$ and $0$ elsewhere, a cofibrant model for $V[0]$ is given by the projective resolution 
$$
\bullet_3 \xrightarrow{\begin{blockarray}{cccc} & \tinytriangleab & \tinytriangleac & \tinytrianglebc\\
\begin{block}{c[ccc]}
\tinytriangle & 1 &  - 1 & 1\\
\end{block}
\end{blockarray}
} \bullet_2 \xrightarrow{\begin{blockarray}{cccc} & \tinytriangleabac & \tinytriangleabbc & \tinytriangleacbc\\
\begin{block}{c[ccc]}
\tinytriangleab & 1 &  - 1 & 0\\
\tinytriangleac & 1 & 0 &  - 1\\
\tinytrianglebc & 0 & 1 &  - 1\\
\end{block}
\end{blockarray}} \bullet_1 \xrightarrow{\begin{blockarray}{cc} & \tinytriangleabacbc\\
\begin{block}{c[c]}
\tinytriangleabac & 1\\
\tinytriangleabbc & 1\\
\tinytriangleacbc & 1\\
\end{block}
\end{blockarray}} \bullet_0.
$$
\subsection{Building cofibrant models}
We begin with a general lemma.% that provides cofibrant chain models.  
\begin{lem} \label{lem:cofibrant}
Suppose that $\{\;\mathcal{U}_i\;\}_{i \in I}$ is a good cover of $X^n$ indexed by a totally-ordered set $I$.  Assume that every graphical configuration space $\Conf(\Gamma, X) \subseteq X^n$ may be obtained as a union of opens drawn from this cover.  Equivalently, for every graph $\Gamma \in \mathcal{G}(n)$, 
$$
\Conf(\Gamma, X) = \bigcup_{i \in S_{\Gamma}} \mathcal{U}_i
$$
where %$S_{\Gamma} \subseteq I$ is the subset
$$
S_{\Gamma} = \{ \mbox{ $i \in I$ so that $\mathcal{U}_i \subseteq \Conf(\Gamma, X)$ } \}.
$$
Then, a cofibrant model for the functor $C_\bullet^{sing} \Conf(-, X)$ is given in degree $p$ by the projective presheaf
$$
\bigoplus_{\substack{i_0 < \cdots < i_p\\ \mathcal{U}_{i_0} \cap \cdots \cap\, \mathcal{U}_{i_p} \neq \emptyset }} (-,g(i_0, \ldots, i_p)),
$$
where $g(i_0, i_1, \ldots, i_p)$ is the largest graph $\Gamma$ for which $\mathcal{U}_{i_r} \subseteq \Conf(\Gamma, X)$ for all $r \in \{0, \ldots, p\}$.  The differentials are given in the usual way, where the entry in row $j_0 < \cdots < j_{p+1}$ and column $i_0 < \cdots < i_{p}$ is $(-1)^l$ if
$$
(i_0 < \cdots < i_{p}) = (j_0 < \cdots < \widehat{j_l} < \cdots < j_{p+1})
$$
for some $l \in \{0, \ldots, p+1\}$, and zero otherwise.
\end{lem}
\begin{proof}
By the nerve lemma, $X^n$ is homotopy equivalent to the \v{C}ech nerve of its good cover $\{\;\mathcal{U}_i\;\}_{i \in I}$, which is an abstract simplicial complex on the vertex set $I$.  For each graph $\Gamma$, the full subcomplex on the vertices $S_\Gamma$ is homotopy equivalent to $\Conf(\Gamma, X)$, once again by the nerve lemma.  By functoriality of the \v{C}ech nerve, the $\mathcal{G}(n)^{op}$-filtration on the simplicial complex model for $X^n$ has the homotopy type of $\Conf(-, X)$, and so the $\mathcal{G}(n)^{op}$-filtered complex produced above is a representative for $C_\bullet^{sing} \Conf(-, X)$ considered as an object of the derived category.
\end{proof}

\begin{rem} \label{rem:wg17}
If $X$ is a simplicial complex, then a good cover of $X^n$ satisfying the hypotheses of Lemma \ref{lem:cofibrant} may be built from the open stars of faces in the usual triangulation of $X^n$.  We give a detailed account of this construction for ordinary, non-graphical configuration space in \cite{WG17}.
\end{rem}

\begin{ex} \label{ex:interval}
Let $X=[0,1]$ be the closed unit interval, and define a good cover of $X^2$
\begin{align*}
\mathcal{U}_1 &= \{(x_1, x_2) \in X^2 \mid x_1 < x_2\} \\
\mathcal{U}_2 &= X^2 \\
\mathcal{U}_3 &= \{(x_1, x_2) \in X^2 \mid x_1 > x_2\}.
\end{align*}
Note that
$$
\Conf(\,\tinyintervalab\;, X) = \mathcal{U}_1 \cup \mathcal{U}_3,
$$
and so the hypotheses of Lemma \ref{lem:cofibrant} are satisfied.
The resulting simplicial complex model for $X^2$ has vertices $\{1,2,3\}$ and edges $\{12,23\}$.  We obtain the following cofibrant model for the singular chains in $\Conf(-,X)$:
$$
\bullet_1 \xrightarrow{\begin{blockarray}{cccc} & \tinyintervalab & \tinyinterval & \tinyintervalab \\
\begin{block}{c[ccc]}
\tinyinterval & -1 &  1 & 0\\
\tinyinterval & 0 & -1 &  1\\
\end{block}
\end{blockarray}} \bullet_0.
$$
\end{ex}
Having found a cofibrant model, it makes sense to ask if there is a smaller cofibrant model.  We describe one common way this can happen.  If row $r$ and column $c$ of $\partial_i$ are both labeled by the same graph $\Gamma$, and if the entry in position $(r,c)$ is $\pm1$, then elementary row and column operations suffice to eliminate all other entries in that row and column.  After this change of basis, we obtain a summand of the form
$$
\bullet_{i+1} \xrightarrow{\begin{blockarray}{cc} & \Gamma \\
\begin{block}{c[c]}
\Gamma & \pm1 \\
\end{block}
\end{blockarray}} \bullet_i,
$$
which is evidently zero in the derived category.  This process of splitting off an acyclic summand from $\partial_i$ is called pruning, and $\partial_i$ is said to be \textbf{prunable}.
We introduce a more general pruning lemma for chain complexes of presheaves on a poset.

A matrix $M$ is said to be \textbf{label-homogenous} if at most one graph appears among its row and column labels.
\begin{lem} \label{lem:prune}
A differential $\partial$ is prunable if it has a label-homogenous submatrix with $1$ as a Smith invariant factor.
\end{lem}
\begin{proof}
Change the basis so that the invariant factor $1$ appears in $\partial$.  We may then make further elementary row and column operations so that this $1$ is the only nonzero entry in its row and column.  %To justify these matrix operations, recall that, for example, a column operation may be thought of as right multiplication by an invertible matrix.  But any invertible matrix corresponds by Lemma \ref{lem:yoneda} to an automorphism of a corresponding direct sum of basic projective presheaves.  In other words, the proof is the same as in usual linear algebra.
After the invariant factor $1$ is alone in its row and column, we may now drop that row and column, since they form an acyclic summand.
\end{proof}
\begin{ex}[Pruning] \label{ex:pruning}
We apply Lemma \ref{lem:prune} to the cofibrant model produced in Example \ref{ex:interval}.  The center column is label-homogeneous, and has $1$ as a Smith factor, so we may prune.  First, change the row basis to get a zero in the lower middle:
$$
\bullet_1 \xrightarrow{\begin{blockarray}{cccc} & \tinyintervalab & \tinyinterval & \tinyintervalab \\
\begin{block}{c[ccc]}
\tinyinterval & -1 &  1 & 0\\
\tinyinterval & -1 & 0 &  1\\
\end{block}
\end{blockarray}} \bullet_0.
$$
Now, change the column basis to get a zero in the upper left:
$$
\bullet_1 \xrightarrow{\begin{blockarray}{cccc} & \tinyintervalab & \tinyinterval & \tinyintervalab \\
\begin{block}{c[ccc]}
\tinyinterval & 0 &  1 & 0\\
\tinyinterval & -1 & 0 &  1\\
\end{block}
\end{blockarray}} \bullet_0.
$$
We are left with the pruned chain model
$$
\bullet_1 \xrightarrow{\begin{blockarray}{ccc} & \tinyintervalab & \tinyintervalab \\
\begin{block}{c[cc]}
\tinyinterval & -1 &  1\\
\end{block}
\end{blockarray}} \bullet_0.
$$
\end{ex}

Finally, we explain how to build a cofibrant model for graphical configuration space in a product.  Note that a model for chains in the product
$$
\Conf(-,X) \times \Conf(-,Y)
$$
is given by the outer tensor product of chain models, and is cofibrant as a chain complex of presheaves on $\mathcal{G}(n) \times \mathcal{G}(n)$.  We may then easily apply $\mathbb{L}(U_n^{op})_!$ using the rule $(U_n^{op})_!(-, (\Gamma', \Gamma'')) = (-,U_n(\Gamma', \Gamma''))$.

\begin{ex} \label{ex:prodmodel}
Using the pruned complex from Example \ref{ex:pruning}, we compute a cofibrant model for singular chains in $\Conf(-,[0,1]^2)$.  Tensoring the pruned model with itself, we obtain
$$
\bullet_2 \xrightarrow{\begin{blockarray}{ccccc} & (\tinyinterval \hspace{-2pt}\times \hspace{-2pt}\tinyintervalab) & (\tinyinterval \hspace{-2pt}\times\hspace{-2pt} \tinyintervalab) & (\tinyintervalab \hspace{-2pt}\times\hspace{-2pt} \tinyinterval) &(\tinyintervalab \hspace{-2pt}\times\hspace{-2pt} \tinyinterval) \\
\begin{block}{c[cccc]}
(\tinyinterval \hspace{-2pt}\times\hspace{-2pt} \tinyinterval) & -1 &  1 & 1 & -1\\
\end{block}
\end{blockarray}} \bullet_1 \xrightarrow{\begin{blockarray}{ccccc} & (\tinyintervalab \hspace{-2pt}\times\hspace{-2pt} \tinyintervalab) & (\tinyintervalab\hspace{-2pt} \times \hspace{-2pt}\tinyintervalab) & (\tinyintervalab\hspace{-2pt} \times\hspace{-2pt} \tinyintervalab) &(\tinyintervalab\hspace{-2pt} \times\hspace{-2pt} \tinyintervalab) \\
\begin{block}{c[cccc]}
(\tinyinterval \hspace{-2pt}\times \hspace{-2pt}\tinyintervalab) & -1 &  0 & 1 & 0\\
(\tinyinterval \hspace{-2pt}\times \hspace{-2pt}\tinyintervalab) & 0 &  -1 & 0 & 1\\
(\tinyintervalab \hspace{-2pt}\times \hspace{-2pt}\tinyinterval) & -1 & 1 & 0 & 0 \\
(\tinyintervalab \hspace{-2pt}\times \hspace{-2pt}\tinyinterval) & 0 & 0 & -1 & 1 \\
\end{block}
\end{blockarray}} \bullet_0.
$$
Applying the union functor $U_2$ to each matrix, we obtain 
$$
\bullet_2 \xrightarrow{\begin{blockarray}{ccccc} & \tinyintervalab & \tinyintervalab & \tinyintervalab &\tinyintervalab  \\
\begin{block}{c[cccc]}
\tinyinterval & -1 &  1 & 1 & -1\\
\end{block}
\end{blockarray}} \bullet_1 \xrightarrow{\begin{blockarray}{ccccc} & \tinyintervalab & \tinyintervalab& \tinyintervalab& \tinyintervalab \\
\begin{block}{c[cccc]}
\tinyintervalab & -1 &  0 & 1 & 0\\
\tinyintervalab & 0 &  -1 & 0 & 1\\
\tinyintervalab & -1 & 1 & 0 & 0 \\
\tinyintervalab & 0 & 0 & -1 & 1 \\
\end{block}
\end{blockarray}} \bullet_0.
$$
This complex is then a chain model for $\Conf(-, [0,1]^2)$ by Theorem \ref{thm:pointwise}.  
\end{ex}
\begin{rem} 
It is good computational hygene to apply Lemma \ref{lem:prune} after every algebraic construction, since the subsequent savings are often considerable.  Moreover, the pruned complexes may be so small as to be recognizable, as is the case in our proof of Theorem~\ref{thm:torus-mod-rotation}.  In that proof, every encountered complex prunes to a sum of complexes drawn from a list of four previously-identified building blocks.
\end{rem}

The next result explains how to compute $H_\bullet(\Conf(n, X) \,;\;\mathbb{Z})$ from a chain model for the functor $\Conf(-,X)$.  Let $K \colon \mathcal{G}(n) \to \Ab$ be the skyscraper functor given by
$$
K\Gamma = \begin{cases}
 \mathbb{Z} & \mbox{ if $\Gamma$ is the complete graph $K_n$} \\
 0 & \mbox{ otherwise. }
\end{cases}
$$

\begin{prop}
If
$$
\cdots \overset{\partial}{\longrightarrow}  \bullet_3 \overset{\partial}{\longrightarrow}  \bullet_2 \overset{\partial}{\longrightarrow} \bullet_1 \overset{\partial}{\longrightarrow}  \bullet_0
$$
is a chain model for $\Conf(-,X) \colon \mathcal{G}^{op} \to \Top$, then
$$
\cdots \overset{K\partial}{\longrightarrow}  \bullet_3 \overset{K\partial}{\longrightarrow}  \bullet_2 \overset{K\partial}{\longrightarrow} \bullet_1 \overset{K\partial}{\longrightarrow}  \bullet_0
$$
is a chain complex of free abelian groups that computes $H_\bullet(\Conf(n, X) \,;\;\mathbb{Z})$.
\end{prop}
\begin{proof}
The functor $K$ coincides with the representable functor $(K_n, -)$.  By Yoneda's lemma, the functor tensor product satisfies $(K_n, -) \otimes_{\mathcal{G}(n)} V \simeq VK_n$, and so we may evaluate at $K_n$ by tensoring with the functor $K$.  Computationally, tensoring a matrix with $K$ amounts to applying $K$ to every entry.
\end{proof}

\begin{ex}
Applying $K$ to the model from Example \ref{ex:prodmodel}, we obtain the chain complex
$$
0 \longrightarrow 
 \mathbb{Z}^4 \xrightarrow{\left[
\begin{array}{cccc}
 -1 & 0 & 1 & 0 \\
 0 & -1 & 0 & 1 \\
 -1 & 1 & 0 & 0 \\
 0 & 0 & -1 & 1 \\
\end{array}
\right]} \mathbb{Z}^4.
$$
It follows that $H_1 \simeq H_0 \simeq \mathbb{Z}$.  This makes sense, since $\Conf(2, [0,1]^2)$ has the homotopy type of a circle.
\end{ex}

\begin{ex}[Triples in a product of capital letters] \label{ex:computation}
We apply the methods of this section to the computation of
$$
H_\bullet \Conf(3, X \times Y)
$$
where $X$ and $Y$ are capital letters considered as one-dimensional simplicial complexes.  We limit our table to the sans serif letters \textsf{X, Y, Z, O}.  In particular, \textsf{Z} $ \cong [0,1]$ and \textsf{O} $ \cong S^1$.

\begin{center}
\begin{tabular}{clllllll}
 & $H_0$ & $H_1$ & $H_2$ & $H_3$ & $H_4$ & $H_5$ & $H_6$ \\
\textsf{X} $\times$ \textsf{Y} & $\mathbb{Z}$ & 0 & 0 & $\mathbb{Z}^{15}$ & $\mathbb{Z}^{230}$ &  &  \\
\textsf{X} $\times$ \textsf{Z} & $\mathbb{Z}$ & 0 & $\mathbb{Z}^{15}$ & $\mathbb{Z}^{46}$ &  &  &  \\
\textsf{X} $\times$ \textsf{O} & $\mathbb{Z}$ &  $\mathbb{Z}^3$& $\mathbb{Z}^{18}$ & $\mathbb{Z}^{77}$ & $\mathbb{Z}^{61}$ &  &  \\
\textsf{Y} $\times$ \textsf{Y} & $\mathbb{Z}$ & 0 & 0 & $\mathbb{Z}^3$ & $\mathbb{Z}^{50}$ &  &  \\
\textsf{Y} $\times$ \textsf{Z} & $\mathbb{Z}$ & 0 & $\mathbb{Z}^3$ & $\mathbb{Z}^{10}$ &  &  &  \\
\textsf{Y} $\times$ \textsf{O} & $\mathbb{Z}$ & $\mathbb{Z}^3$ & $\mathbb{Z}^6$ & $\mathbb{Z}^{17}$ & $\mathbb{Z}^{13}$ &  &  \\
\textsf{Z} $\times$ \textsf{Z} & $\mathbb{Z}$ & $\mathbb{Z}^3$ & $\mathbb{Z}^2$ &  &  &  &  \\
\textsf{Z} $\times$ \textsf{O} & $\mathbb{Z}$ & $\mathbb{Z}^6$ & $\mathbb{Z}^{11}$ & $\mathbb{Z}^6$ &  &  &  \\
\textsf{O} $\times$ \textsf{O} & $\mathbb{Z}$ & $\mathbb{Z}^6$ & $\mathbb{Z}^{14}$ & $\mathbb{Z}^{14}$ & $\mathbb{Z}^{5}$ &  & 
\end{tabular}
\end{center}
In the cases where one of the factors is \textsf{X} or \textsf{Y}, this computation is new, since the only known models for configuration space in a singular two-dimensional space \cite[Theorems 1.3 and 1.12]{WG17} are impractically large.  This table was computed using Sage \cite{sage}.
\end{ex}

\section{Proof of Theorem \ref{thm:torus-mod-rotation}} \label{sec:torus-mod-circle-proof}
We apply systematically the method from \S\ref{sec:method}  to prove Theorem \ref{thm:torus-mod-rotation}.  Since the action of the torus is free on configuration space, the $T$-equivariant homology coincides with the homology of the quotient.  This allows us to work directly with the quotient.  

Let $T= \mathbb{R}/\mathbb{Z}$ be a torus of rank one.  Write $\Delta \mathbb{R} = (x,x,x) \subset \mathbb{R}^3$ for the main diagonal of Euclidean space, and similarly $\Delta \mathbb{Z}^3 \subset \mathbb{Z}^3$.  We start by building an explicit model for the functor
$$
\Conf(-,T)/T \colon \mathcal{G}(3)^{op} \to \Top.
$$
Since $T^3/T \cong \mathbb{R}^3/(\Delta \mathbb{R} + \mathbb{Z}^3)$, the universal cover of $T^3/T$ is $\mathbb{R}^3/\Delta\mathbb{R}$.  The rank two lattice $\mathbb{Z}^3/\Delta\mathbb{Z}$ acts on this cover, and we choose a regular hexagon as the fundamental domain.  The boundary points of the hexagon give configurations where one point is exactly halfway between the other two.  The vertices of the hexagon give configurations where all three points are equally spaced.

We depict thirteen open sets in $\mathbb{R}^3/\Delta\mathbb{R}$ that together give a good cover of $T^3/T$.  The black hexagonal lattice divides the plane into an infinite number of copies of the fundamental domain, and the dashed lines indicate the diagonals to be removed.
\begin{center}
\includegraphics[scale=2.85]{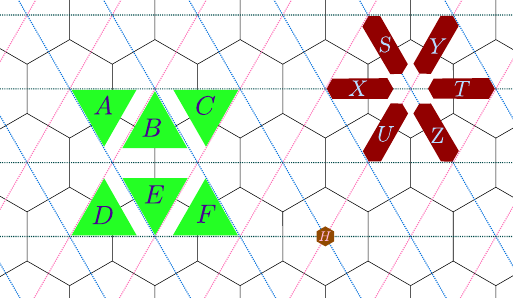}
\end{center}
The \v{C}ech nerve of this cover has facets
\begin{align*}
\{AEC, BDF, &AHX, AHS, BHS, BHY, CHY, CHT, DHX, DHU, EHU, EHZ, \\
&FHT, FHZ, BFSZ, AESZ, BDUY, CEUY, ACTX, DFTX\}.
\end{align*}
Since every graphical configuration space is available as a union of sets drawn from this cover,  Lemma \ref{lem:cofibrant} gives a cofibrant chain model.
Pruning this model using Lemma \ref{lem:prune}, we obtain the complex $M_\bullet$, given by
$$
\bullet_2 \xrightarrow{\begin{blockarray}{cccc} & \tinytriangleacbc & \tinytriangleabbc & \tinytriangleabac   \\
\begin{block}{c[ccc]}
\tinytriangle & 0 &  0 & 0 \\
\end{block}
\end{blockarray}} \bullet_1 \xrightarrow{\begin{blockarray}{cccc} & \tinytriangleabacbc & \tinytriangleabacbc \\
\begin{block}{c[ccc]}
\tinytriangleacbc & 0 &  1  \\
\tinytriangleabbc & 0 &  1\\
\tinytriangleabac & 0 & 1 \\
\end{block}
\end{blockarray}} \bullet_0.
$$
The complex splits as a direct sum of three complexes 
$$
M_\bullet = A_\bullet \oplus B_\bullet \oplus C_\bullet
$$
where $A_\bullet = \tinytriangleabacbc[0]$, $C_\bullet = \tinytriangle[2]$, and $B_\bullet$ is given by
$$
\bullet_1 \xrightarrow{\begin{blockarray}{ccc} & \tinytriangleabacbc \\
\begin{block}{c[cc]}
\tinytriangleacbc & 1  \\
\tinytriangleabbc & 1\\
\tinytriangleabac & 1 \\
\end{block}
\end{blockarray}} \bullet_0.
$$

Before entering the main calculation, we simplify notation.
\begin{defn}
If $E_\bullet, F_\bullet \in D_+(\mathcal{A})$ are complexes, write
$$
E_\bullet \odot F_\bullet = \mathbb{L}(U_3^{op})_! (E_\bullet \boxtimes F_\bullet).
$$
\end{defn}
%Since the union functor is commutative and associative, we have
%$$
%(E_\bullet \odot F_\bullet) \simeq (F_\bullet \odot E_\bullet)
%$$
%and
%$$
%(E_\bullet \odot (F_\bullet \odot G_\bullet)) \simeq ((E_\bullet \odot F_\bullet) \odot G_\bullet)
%$$
%for any triple of complexes $E_\bullet, F_\bullet, G_\bullet \in D^b(\mathcal{A})$.  Consequently, we will no longer need to include so many parentheses in calculations.
\noindent
Observe that the functor $\odot$ distributes over direct sums, and is commutative and associative up to quasi-isomorphism.  Using the new notation, the aim of this section is to compute
$$
H_\bullet (M_\bullet^{\odot r}) = H_\bullet \; (A_\bullet \oplus B_\bullet \oplus C_\bullet)^{\odot r}.
$$
\begin{lem} \label{lem:multiplication}
We have the following multiplication table for $\odot$:
\begin{center}
\begin{tabular}{l|cccc}
$\odot$ & $A_\bullet$ & $B_\bullet$ & $C_\bullet$ & $D_\bullet$ \\
 \hline
$A_\bullet$ & $A_\bullet$ & $A_\bullet[1]^{\oplus 2}$ & $A_\bullet[2]$ & $A_\bullet[2]^{\oplus 3}$ \\
$B_\bullet$ & $A_\bullet[1]^{\oplus 2}$ & $D_\bullet \oplus A_\bullet[2]$ & $B_\bullet[2]$ & $D_\bullet[1] \oplus A_\bullet[3]^{\oplus 3}$ \\
 $C_\bullet$& $A_\bullet[2]$ &  $B_\bullet[2]$ & $C_\bullet[2]$ & $D_\bullet[2]$,%\\
%$D_\bullet$ & $A_\bullet[2]^{\oplus 3}$ & $D_\bullet[1] \oplus A_\bullet[3]^{\oplus 3}$ & $D_\bullet[2]$,
\end{tabular}
\end{center}
where $D_\bullet = \tinytriangleacbc[2] \oplus \tinytriangleabbc[2] \oplus \tinytriangleabac[2]$.
\end{lem}
\begin{proof}
This is proved by pruning; see Lemma \ref{lem:prune}.
\end{proof}
\noindent
For notational convenience, we introduce a dummy variable $s$ that stands for ``shift.''  Given a polynomial in $s$ with nonnegative coefficients,
$$
e(s) = e_0 + e_1 s + e_2 s^2 + \cdots + e_k s^k \in \mathbb{N}[s],
$$
define
$$
E_\bullet^{\oplus e(s)} = E_\bullet^{\oplus e_0} \oplus E_\bullet[1]^{\oplus e_1} \oplus E_\bullet[2]^{\oplus e_2} \oplus \cdots \oplus E_\bullet[k]^{\oplus e_k}.
$$
In this way, we combine shifts and multiplicities into a single operation.
If $a, b, c, d \in \mathbb{N}[s]$, then define
$$
P_\bullet(a, b, c, d) =  (A_\bullet^{\oplus a}  \oplus B_\bullet^{\oplus b} \oplus C_\bullet^{\oplus c} \oplus D_\bullet^{\oplus d} ).
$$
We obtain the following rephrasing of Lemma \ref{lem:multiplication}.
\begin{cor}
%The complex $(A_\bullet \oplus B_\bullet \oplus C_\bullet) \odot P_\bullet(a, b, c, d)$ is quasi-isomorphic to a direct sum of shifts of the four complexes $A_\bullet, B_\bullet, C_\bullet, D_\bullet$.  
If
$$
\left[
\begin{array}{cccc}
 s^2+2 s+1 & s^2+2 s & s^2 & 3 s^3+3 s^2 \\
 0 & s^2 & s^2 & 0 \\
 0 & 0 & s^2 & 0 \\
 0 & 1 & 0 & s + s^2\\
\end{array}
\right] \cdot
\left[
\begin{array}{c}
 a \\
 b \\
 c \\
 d \\
\end{array}
\right] = \left[
\begin{array}{c}
 a' \\
 b' \\
 c' \\
 d' \\
\end{array}
\right],
$$
then
$$
(A_\bullet \oplus B_\bullet \oplus C_\bullet) \odot P_\bullet(a,b,c,d) = P_\bullet(a',b',c',d').
$$
\end{cor}
\noindent
As a consequence, the product
$$
\left[
\begin{array}{cccc}
 s^2+2 s+1 & s^2+2 s & s^2 & 3 s^3+3 s^2 \\
 0 & s^2 & s^2 & 0 \\
 0 & 0 & s^2 & 0 \\
 0 & 1 & 0 & s + s^2 \\
\end{array}
\right]^{r-1} \cdot
\left[
\begin{array}{c}
 1 \\
 1 \\
 1 \\
 0 \\
\end{array}
\right]
$$
gives a complete description of the object $M_\bullet^{\odot r} \in D^b(\mathcal{A})$.  Each of $A_\bullet$ and $B_\bullet$ contribute $\mathbb{Z}[0]$ in homology after evaluation at the complete graph $\tinytriangleabacbc$, and so
$$
\left[
\begin{array}{cccc}
 1 & 1 & 0 & 0\\
\end{array}
\right] \cdot
\left[
\begin{array}{cccc}
 s^2+2 s+1 & s^2+2 s & s^2 & 3 s^3+3 s^2 \\
 0 & s^2 & s^2 & 0 \\
 0 & 0 & s^2 & 0 \\
 0 & 1 & 0 & s + s^2 \\
\end{array}
\right]^{r-1} \cdot
\left[
\begin{array}{c}
 1 \\
 1 \\
 1 \\
 0 \\
\end{array}
\right] = \sum_{p = 0}^{\infty} \beta_p s^p,
$$
where
$$
\beta_p = \mathrm{rk}_{\mathbb{Z}} H_p \Conf(3, T^r)/T^r.
$$
The formula for $\beta_p$ appearing in Theorem \ref{thm:torus-mod-rotation} may be deduced from this generating function by a routine induction on $r$.

\section{Proof of Theorem \ref{thm:stable}} \label{sec:stable-proof}
In this section, we prove Theorem \ref{thm:stableformula}, which is a strengthened version of Theorem \ref{thm:stable}.
\begin{lem} \label{lem:split}
Let $\mathcal{A}$ be an abelian category with enough projectives, and let $K_{\bullet} \in \Ch_\bullet(\mathcal{A})$ be a chain complex so that
\begin{itemize}
\item $H_i(K_\bullet) = 0$ unless $i \in \{0, k, 2k, \ldots \}$ and
\item for all $i$, the projective dimension of $H_i(K_\bullet)$ is at most $k-1$;
\end{itemize}
then $K_\bullet$ is quasi-isomorphic to its homology.
\end{lem}
\begin{proof}
Choose a projective resolution $P_\bullet \to H_0(K_\bullet)$ that vanishes past degree $k-1$.  Since $K_\bullet$ is exact in degrees below $k$, projectivity lets us construct a map $P_\bullet \to K_\bullet$ for all degrees below $k$ in the same manner as if $K_\bullet$ were a resolution.  For all degrees $k$ and above, the complex $P_\bullet$ vanishes, and so the map must be zero anyway.  This construction produces a map $P_\bullet \to K_\bullet$ that induces an isomorphism on $H_0$.

A similar map may be constructed from a projective resolution of $H_k(K_\bullet)$ to $K_\bullet$, and so on, building a quasi-isomorphism from the direct sum of these resolutions to $K_\bullet$.  But this direct sum is evidently quasi-isomorphic to its homology, and the lemma is proved.
\end{proof}

According to foundational work of Brieskorn \cite{Brieskorn73}, the cohomology $H^\bullet(\Conf(\Gamma, \mathbb{C}), \mathbb{Z})$ is generated as a $\mathbb{Z}$-algebra by the differential forms
$$
\omega_{ij} = \frac{1}{2\pi \sqrt{-1}} \frac{dz_i - dz_j}{z_i - z_j},
$$
and every degree is a free abelian group.
Subsequent work of Orlik and Solomon \cite{OS80} determined the ideal of relations between these forms.  The relations happen on-the-nose at the cochain level, and do not require the introduction of nonzero coboundaries.  It follows that the functor
$$
\Gamma \mapsto \langle \omega_{ij} \rangle_{i\sim_{\Gamma} j} 
$$
sending a graph to its subalgebra of $\Omega^\bullet \Conf(K_n, \mathbb{C})$ provides a cochain model for the functor $C^\bullet \Conf(-,\mathbb{C})$.  The differential for this model is zero, so we might say that graphical configuration space in $\mathbb{C}$ is a ``formal'' $\mathcal{G}(n)^{op}$-space.  Write 
$$
F_i \colon \Gamma \mapsto H_{i} \Conf(-, \mathbb{C})
$$
for the homology presheaf in degree $i$, so that formality and torsion-freeness gives a quasi-isomorphism
$$
C_\bullet^{sing} \Conf(-,\mathbb{C}) \simeq  \bigoplus_{i =0}^{n-1} F_i[i].
$$

The cohomology $H^\bullet(\Conf(\Gamma, \mathbb{C}^p), \mathbb{Z})$ is simply a regrading of the case $p=1$ by work of de Longueville and Schultz \cite{dLS01}, who provide a combinatorial presentation for this ring that depends only on the intersection combinatorics of the linear arrangement associated to $\Gamma$, and not on the specific value of $p\geq 1$.  All of the cohomology is torsion-free, and concentrated in degrees that are multiples of $k=2p-1$ because the generating classes $\omega_{ij}$ sit in degree $2p-1$.  Integration of  these classes gives higher-dimensional analogs of winding numbers.

The following result remains true without its hypothesis, but our proof is structured around understanding the case of large $p$ first.  For the improved statement, see Corollary~\ref{cor:decomposition}.
\begin{prop} \label{prop:largep}
For $p$ large relative to $n$, we have a quasi-isomorphism
$$
C_\bullet^{sing} \Conf(-,\mathbb{C}^p) \simeq \bigoplus_{i =0}^{n-1} F_i[ki],
$$
where $k=2p-1$.
\end{prop}
\begin{proof}
By the result of de Longueville and Schultz \cite{dLS01}, for every $i$ we have an isomorphism of presheaves $H_{ki} \Conf(-,\mathbb{C}) \cong F_i$.  This presheaf takes values in free abelian groups, so its projective dimension does not exceed the dimension of the nerve of $\mathcal{G}(n)^{op}$, which is $\binom{n}{2}$.  The hypotheses of Lemma \ref{lem:split} are then met by any $p$ for which $\binom{n}{2} \leq k-1$, and this concludes the proof. 
\end{proof}

As in \S\ref{sec:method}, write $\odot$ for the commutative, associative operation $\mathbb{L}(U_n^{op})_!$, and note the compatibility with shifts:
$$
F[a] \odot G[b] \simeq (F \odot G)[a+b],
$$
for all presheaves $F, G$.  We state and prove an orthogonality property of the presheaves $F_i$.

\begin{thm} \label{thm:orthogonal}
For all $i, j \in \{0, \ldots, n-1\}$, we have quasi-isomorphisms
$$
F_i \odot F_j \simeq \begin{cases} F_i[i] &\mbox{if } i=j \\ 
0 & \mbox{if } i \neq j, \end{cases}
$$
where $F_i \colon \mathcal{G}(n)^{op} \to \Ab$ is the homology presheaf $H_i \Conf(-,\mathbb{C})$.
\end{thm}
\begin{proof}
Assume that is $p$ large enough so that the conclusion of Proposition \ref{prop:largep} holds.  Consider the big product
\begin{equation} \label{eq:prod}
\left(\bigoplus_{i =0}^{n-1} F_i[ik] \right)^{\odot p} .
\end{equation}
As explained in \S\ref{sec:method}, Theorem \ref{thm:pointwise} claims this product is a model for $C_\bullet \Conf(-, \mathbb{C}^{p^2})$, and so (\ref{eq:prod}) is quasi-isomorphic to
\begin{equation} \label{eq:little}
\bigoplus_{i =0}^{n-1} F_i[ik']
\end{equation}
again by Proposition \ref{prop:largep}, where $k'=2p^2-1$.  On the other hand, expanding (\ref{eq:prod}) by the multinomial theorem, we obtain
$$
\bigoplus_{\substack{m_0, \ldots, m_{n-1} \\ m_0 + \cdots + m_{n-1} = p}} \frac{p!}{m_0! \cdots m_{n-1}!} \cdot F_0[0]^{\odot m_0} \odot F_1[k]^{\odot m_1} \odot F_2[2k]^{\odot m_2} \odot \cdots \odot F_{n-1}[(n-1)k]^{\odot m_{n-1}}.
$$
Despite the explosion of terms and large multiplicities, the quasi-isomorphism with (\ref{eq:little}) shows that this complex has little homology: only $F_i$ in degree $ik'$ for each $i$.

We consider ``cross-terms,'' by which we mean terms where there are two or more non-vanishing $m_i$.  The multinomial coefficient for a cross-term has size at least $p$.  If we assume that $p$ is larger than the total rank of the abelian groups appearing in the presheaves $F_i$, then we claim that any cross-term must vanish in the derived category.  Indeed, a non-vanishing object in the derived category must contribute to some degree of homology, and any such contribution causes a contradiction after multiplication by the huge coefficient.

By degree considerations, we see that the remaining terms---the non-cross-terms---satisfy
$$
(F_i[ik])^{\odot p} \simeq F_i[ik'],
$$
and so
$$
F_i^{\odot p}[pi(2p-1)] \simeq F_i^{\odot p}[pik] \simeq (F_i[ik])^{\odot p} \simeq F_i[ik'] \simeq F_i[i(2p^2-1)].
$$
Since $i(2p^2-1) - pi(2p-1) = i(p-1)$, we deduce $F_i^{\odot p} \simeq F_i[i(p-1)]$.
Now compute:
\begin{align*}
(F_i \odot F_i)[2i(p-1)] &\simeq F_i[i(p-1)] \odot F_i[i(p-1)] \\
&\simeq F_i^{\odot p} \odot F_i^{\odot p} \\
&\simeq F_i^{\odot 2p} \\
&\simeq F_i[i(2p-1)],
\end{align*}
from which we conclude that $F_i \odot F_i \simeq F_i[i]$, as required.  Similarly, if $i \neq j$,
\begin{align*}
(F_i \odot F_j)[i(p-1) + j(p-1)] &\simeq F_i[i(p-1)] \odot F_j[j(p-1)] \\
&\simeq F_i^{\odot p} \odot F_j^{\odot p} \\
&\simeq (F_i^{\odot (p - 1)} \odot F_j) \odot (F_j^{\odot (p - 1)} \odot F_i) \\
&\simeq 0 \odot 0.
\end{align*}
since both of these factors are cross-terms.
\end{proof}

\begin{cor} \label{cor:decomposition}
Proposition \ref{prop:largep} holds for all $p \geq 1$:
$$
C_\bullet^{sing} \Conf(-,\mathbb{C}^p) \simeq \bigoplus_{i =0}^{n-1} F_i[ki].
$$
\end{cor}
\begin{proof}
We already observed the fact for $p=1$ from work of Brieskorn \cite{Brieskorn73}.  The result then follows from Theorem \ref{thm:pointwise} and Theorem \ref{thm:orthogonal} using $\mathbb{C}^p = \mathbb{C} \times \cdots \times \mathbb{C}$.
\end{proof}

\begin{thm} \label{thm:stableformula}
If $X$ is a Hausdorff space, and 
$$
M_\bullet \colon \mathcal{G}(n)^{op} \to \Ch_\bullet(\mathbb{Z})
$$
is a chain model for $\Conf(-,X)$, then
$$
H_{mp+b} \Conf(-, X \times \mathbb{C}^p) = \bigoplus_{i=0}^{n-1} H_{(m-2i)p+(b+i)} \left( M_\bullet \odot F_i \right).
$$
If $X$ is homeomorphic to a complement $A \setminus B$ for $A$ a $d$-dimensional simplicial complex and $B$ a subcomplex, then
$$
H_{mp+b} \Conf(-, X \times \mathbb{C}^p) \cong \begin{cases} 0 &\mbox{if $m$ is odd} \\ 
H_{m/2+b} \left( M_\bullet \odot F_{m/2} \right) & \mbox{if $m$ is even} \end{cases}
$$
once $2p > \max \left\{ b+m/2+1, \binom{n}{2} + nd -b -m/2 -1 \right\}$.
\end{thm}
\begin{proof}
By Theorem \ref{thm:pointwise} and Lemma \ref{lem:commutes}, we have a quasi-isomorphism
$$
C_\bullet \Conf(-,X \times \mathbb{C}^p) \simeq (C_\bullet \Conf(-,X)) \odot (C_\bullet \Conf(-,\mathbb{C}^p)).
$$
Substituting the model for $X$ and the model for $\mathbb{C}^p$ in Corollary \ref{cor:decomposition}, we obtain
$$
C_\bullet \Conf(-,X \times \mathbb{C}^p) \simeq M_\bullet \odot \left( \bigoplus_{i =0}^{n-1} F_i[(2p-1)i] \right).
$$
For any $p$,
\begin{align*}
H_{mp+b} \Conf(-,X \times \mathbb{C}^p) &\cong H_{mp+b} \left( M_\bullet \odot \left( \bigoplus_{i =0}^{n-1} F_i[(2p-1)i] \right) \right) \\
&\cong \bigoplus_{i =0}^{n-1} H_{mp+b} \left(  M_\bullet \odot F_i\right)[(2p-1)i]  \\
&\cong \bigoplus_{i =0}^{n-1} H_{(mp+b)-(2p-1)i} \left(  M_\bullet \odot F_i\right),
\end{align*}
from which the first claim follows, since $(mp+b)-(2p-1)i = p(m-2i) +(b+i)$.

We turn our attention to the second claim.  The open stars of the $B$-avoiding faces of the product complex $A^n$ form an open cover that satisfies the hypotheses of Lemma~\ref{lem:cofibrant}, proving that $\Conf(-,X)$ has a cofibrant chain model concentrated in degrees up to $\dim A^n = nd$.  Since the projective dimension of $F_i$ is at most $\binom{n}{2}$, we may assume that the complex $M_\bullet \odot F_i$ is concentrated in degrees up to $\binom{n}{2} + nd$.

All non-vanishing homology must appear in the intersection
$$
\{p(m-2i)+b+i \, | \, 0 \leq i \leq n-1 \} \; \cap \; \left\{0, \ldots, \binom{n}{2} + nd\right\}.
$$
As $p$ gets larger, this first set spreads out, until its only intersection with the second set happens when $m=2i$.  Specifically, this occurs when the adjacent cases, $i = m/2 \pm 1$, leave the interval $\left[0,  \binom{n}{2} + nd\right]$:
$$
\begin{array}{ccc}
-2p+b+m/2+1 < 0 & \mbox{ and } & \binom{n}{2} + nd < 2p + b + m/2 + 1,
\end{array}
$$
which happens once $2p > \max \left\{ b+m/2+1, \binom{n}{2} + nd -b -m/2 -1 \right\}$.
Under this assumption, the only summand that contributes homology has $i=m/2$, giving the result.
\end{proof}

\section{Proof of Theorems \ref{thm:pointwise} and \ref{thm:main}}
We focus on proving the harder result, Theorem \ref{thm:main}.  We deduce Theorem \ref{thm:pointwise} in \S\ref{sec:pointwise_proof}.
\label{sec:main_proof}
\subsection{Organization of the proof}
Write $\Pi \colon \GI_2^{op} \to \Top_{G \times H}$ for the functor called $\Conf(-, X) \times \Conf(-,Y)$ in the introduction.  We build the weak equivalence from Theorem~\ref{thm:main} in two steps
$$
\mathbb{L}(U^{op})_! \Pi \overset{\alpha}{\longrightarrow} (U^{op})_! \Pi \overset{\beta}{\longrightarrow} \Conf(-, X\times Y).
$$
The natural transformation $\alpha$ is the comparison map from the homotopy left Kan extension to the usual left Kan extension.  We show that every component $\alpha_\Gamma$ is a weak equivalence using a theorem of Dugger-Isaksen about the homotopy colimits of open covers.
%https://arxiv.org/abs/math/0111287

We define the natural transformation $\beta$ in terms of its components $\beta_\Gamma$ for every $\Gamma \in \GI$.  We check that the required squares commute, and that every $\beta_\Gamma$ is a homeomorphism.

%Both calculations rely on a pointwise description of the left Kan, and so we begin there.

\subsection{Evaluation of $\mathbb{L} (U^{op})_! \Pi$ and $(U^{op})_! \Pi$ at $\Gamma$}
We employ a general formula for left Kan extensions in terms of colimits.  %The case of $(U^{op})_! \Pi$ is similar, but with $\colim$ in place of $\hocolim$.
In the case of $(U^{op})_! \Pi$, this formula reads
$$
[(U^{op})_! \Pi](\Gamma) = \underset{U^{op}/\Gamma}{\colim} \;\; \left[(U^{op}/\Gamma) \longrightarrow \GI_2^{op} \overset{\Pi}{\longrightarrow} \Top_{G \times H} \right].
$$
The case of $\mathbb{L} (U^{op})_! \Pi$ is identical, but with $\hocolim$ in place of $\colim$.

We have used standard notation for the comma category $(U^{op}/\Gamma)$, but we anyhow give the details of its construction.  For concreteness of exposition we describe the opposite category $(U^{op}/\Gamma)^{op} = (\Gamma/U)$.

An object of $\Gamma/U$ is a pair $((\Gamma', \Gamma'') \, ; \, \Gamma \overset{\psi}{\to} U(\Gamma', \Gamma''))$ of an object $(\Gamma', \Gamma'') \in \GI_2$ and a graph injection $\psi$ from the fixed graph $\Gamma$ to the union of $\Gamma' $ and $\Gamma''$.  A morphism in $\Gamma/U$ is a $\GI_2$ morphism making the evident triangle commute.  Recall that the poset $\mathcal{P}(\Gamma)$ is defined as the full subposet of $\mathcal{G}(\Gamma) \times \mathcal{G}(\Gamma)$ on the objects $(\Gamma', \Gamma'')$ with $U(\Gamma', \Gamma'') = \Gamma$.  The comma category $\Gamma/U$ admits a functor from the poset $\mathcal{P}(\Gamma)$
$$
\begin{array}{rcccl}
\Phi & \colon & \mathcal{P}(\Gamma) &\to& (\Gamma/U) \\
&& (\Gamma', \Gamma'') &\mapsto& ((\Gamma', \Gamma'') \, ; \, \Gamma \overset{\tiny\mbox{id}}{\longrightarrow} U(\Gamma', \Gamma'')),
\end{array}
$$
whose image includes all objects of $\Gamma/U$ for which the graph injection $\psi$ is the identity.
To prove the pointwise result, Theorem \ref{thm:pointwise}, we must factor this functor through the smaller category $(\Gamma/U_n)$
$$
\Phi \colon \mathcal{P}(\Gamma) \xrightarrow{\Phi_a} (\Gamma/U_n) \to (\Gamma/U).
$$
We prove a result about $\Phi_a$, and then prove the same result about $\Phi$.  The proofs are similar, so Proposition \ref{prop:initiala} can be thought of as a warm-up for Proposition \ref{prop:initial}.
\begin{prop} \label{prop:initiala}
The functor $\Phi_a$ is homotopy initial.
\end{prop}
\begin{proof}
For every $\zeta \in (\Gamma/U_n)$, we must show that the category $(\Phi_a/\zeta)$ is contractible.  We do this by producing a terminal object.  Specifically, if
$$
\zeta = ((Z', Z'') \in \mathcal{G}(n) \times \mathcal{G}(n) \, ; \, \Gamma \subseteq U_n(Z', Z'')),
$$
then the terminal object of $(\Phi_a/\zeta)$ is given by
$$
t = ((\Gamma \cap Z', \Gamma \cap Z'') \in \mathcal{P}(\Gamma) \,;\; \Phi_a(\Gamma \cap Z', \Gamma \cap Z'') \overset{\tau}{\to} \zeta)
$$
where $\tau$ is the $(\Gamma/U_n)$ morphism induced by the pair of graph inclusions 
$$
(\tau', \tau'') \colon (\Gamma \cap Z', \Gamma \cap Z'') \subseteq (Z', Z'').
$$
Note that $U_n(\Gamma \cap Z', \Gamma \cap Z'') = \Gamma \cap U_n(Z', Z'') = \Gamma$ by distributivity, so we do have $(\Gamma \cap Z', \Gamma \cap Z'') \in \mathcal{P}(\Gamma)$, and the triangle of inclusions
$$
\begin{tikzcd}
& U_n(\Gamma \cap Z', \Gamma \cap Z'') \ar{dd}{U_n(\tau', \tau'')}\\
\Gamma \ar{ru}{\tiny\mbox{id}} \arrow[swap]{rd}{\subseteq} &  \\
 & U_n(Z', Z''),
\end{tikzcd}
$$
commutes vacuously in the poset $\mathcal{G}(n)$, so the pair $(\tau', \tau'')$ produces a valid morphism of $(\Gamma/U_n)$.

Observe that each hom-set of $(\Phi_a/\zeta)$ has size at most one, since $\mathcal{P}(\Gamma)$ is a poset.  Consequently, $t \in (\Phi_a/\zeta)$ is terminal if it admits a map from every other $s \in (\Phi_a/\zeta)$, since this map will be unique automatically.  Suppose
$$
s = ((\Gamma', \Gamma'') \in \mathcal{P}(\Gamma) \,;\; \Phi_a(\Gamma', \Gamma'') \overset{\sigma}{\to} \zeta)
$$
is some other object of $(\Phi_a/\zeta)$, where the morphism $\sigma$ is induced by two graph inclusions
$$
(\sigma', \sigma'') \colon (\Gamma', \Gamma'') \subseteq (Z', Z'').
$$
From this we see that $\Gamma' \subseteq Z'$ and $\Gamma'' \subseteq Z''$.  On the other hand, $U_n(\Gamma', \Gamma'') = \Gamma$, so $\Gamma' \subseteq \Gamma$ and $\Gamma'' \subseteq \Gamma$.  It follows that there are inclusions
$$
(\varepsilon', \varepsilon'') \colon (\Gamma', \Gamma'') \subseteq (\Gamma \cap Z', \Gamma \cap Z'').
$$
Once again, the commutation of the resulting triangle is vacuous, and so the pair $(\varepsilon', \varepsilon'')$ induces a $(\Phi_a/\zeta)$ morphism to $t$.
%and so we have a commuting triangle of inclusions
%$$
%\begin{tikzcd}
%& U_n(\Gamma', \Gamma'') \ar{dd}{U_n(\varepsilon', \varepsilon'')}\\
%\Gamma \ar{ru}{\tiny\mbox{id}} \arrow[swap]{rd}{\subseteq} &  \\
% & U_n(\Gamma \cap Z', \Gamma \cap Z''),
%\end{tikzcd}
%$$
%showing that the pair $(\varepsilon', \varepsilon'')$ defines a valid morphism in $(\Gamma/\zeta)$.
\end{proof}
\begin{prop} \label{prop:initial}
The functor $\Phi$ is homotopy initial.
\end{prop}
\begin{proof}
We imitate the proof of Proposition \ref{prop:initiala}, showing that for every $\zeta \in (\Gamma/U)$, the category $(\Phi/\zeta)$ has a terminal object.  Suppose
$$
\zeta = ((Z', Z'') \in \GI_2 \, ; \, \Gamma \xrightarrow{\psi} U(Z', Z'')),
$$
and, relabeling the (shared) vertex set of the graphs $Z', Z''$ if necessary, assume that the underlying injection of $\psi$
$$
\psi \colon \{1, \ldots, n\} \to U(Z',Z'')
$$
is an inclusion.  (Relabeling has the effect of replacing $\zeta$ with an isomorphic object, which evidently preserves the homotopy type of the slice category $(\Phi/\zeta)$.  We include this assumption only so that we can keep using the symbol $\cap$ for fiber products.)  We show that the terminal object of $(\Phi/\zeta)$ is given by
$$
t = ((\Gamma \cap Z', \Gamma \cap Z'') \in \mathcal{P}(\Gamma) \,;\; \Phi(\Gamma \cap Z', \Gamma \cap Z'') \overset{\tau}{\to} \zeta)
$$
where $\tau$ is the $(\Gamma/U)$ morphism induced by the pair of graph inclusions 
$$
(\tau', \tau'') \colon (\Gamma \cap Z', \Gamma \cap Z'') \subseteq (Z', Z'').
$$
Again, $U(\Gamma \cap Z', \Gamma \cap Z'') = \Gamma \cap U(Z', Z'') = \Gamma$, so $(\Gamma \cap Z', \Gamma \cap Z'') \in \mathcal{P}(\Gamma)$.  The triangle 
$$
\begin{tikzcd}
& U(\Gamma \cap Z', \Gamma \cap Z'') \ar{dd}{U(\tau', \tau'')}\\
\Gamma \ar{ru}{\tiny\mbox{id}} \arrow[swap]{rd}{\psi} &  \\
 & U(Z', Z''),
\end{tikzcd}
$$
commutes because all three maps are inclusions.%, so the pair $(\tau', \tau'')$ produces a valid morphism of $(\Gamma/U)$.

Observe that each hom-set of $(\Phi/\zeta)$ has size at most one, since $\mathcal{P}(\Gamma)$ is a poset.  Consequently, $t \in (\Phi/\zeta)$ is terminal if it admits a map from every other $s \in (\Phi/\zeta)$.  Suppose
$$
s = ((\Gamma', \Gamma'') \in \mathcal{P}(\Gamma) \,;\; \Phi(\Gamma', \Gamma'') \overset{\sigma}{\to} \zeta)
$$
is some other object of $(\Phi/\zeta)$, where the morphism $\sigma$ is induced by a pair of graph injections
$$
(\sigma', \sigma'') \colon (\Gamma', \Gamma'') \hookrightarrow (Z', Z'')
$$
with the same underlying function---also called $\sigma$---on vertices.  Since $\sigma$ is a morphism in $(\Gamma/U)$, we have a commuting triangle
$$
\begin{tikzcd}
& U(\Gamma', \Gamma'') \ar{dd}{U(\sigma', \sigma'')}\\
\Gamma \ar{ru}{\tiny\mbox{id}} \arrow[swap]{rd}{\psi} &  \\
 & U(Z', Z''),
\end{tikzcd}
$$
which forces $U(\sigma', \sigma'')$, whose underlying function is $\sigma$, to match the standard inclusion $\psi \colon \{1, \ldots, n\} \subseteq U(Z', Z'')$.  From this we see that $\Gamma' \subseteq Z'$ and $\Gamma'' \subseteq Z''$.  Moreover, since $\Gamma', \Gamma'' \subseteq \Gamma$, there are inclusions
$$
(\varepsilon', \varepsilon'') \colon (\Gamma', \Gamma'') \subseteq (\Gamma \cap Z', \Gamma \cap Z''),
$$
and they give a $(\Phi/\zeta)$ map because the required triangle consists of inclusions.%, and so $(\varepsilon', \varepsilon'')$ induces a $(\Phi/\zeta)$ morphism to $t$.
\end{proof}

\begin{cor} \label{cor:hocolim_restrict}
Homotopy colimits of shape $(U^{op}/\Gamma)$ may be evaluated by first restricting to the poset $\mathcal{P}(\Gamma)^{op}$.  Explicitly, the natural map
$$
\underset{\mathcal{P}(\Gamma)^{op}}{\hocolim} \;\; (\Phi^{op})\hspace{-1pt}^*F \overset{\sim}{\longrightarrow}
\underset{U^{op}/\Gamma}{\hocolim} \;\; F
$$
is a weak equivalence for any functor $F \colon (U^{op}/\Gamma) \to \Top_{G \times H}$.
\end{cor}
\begin{proof}
Since $\Phi$ is homotopy initial, $\Phi^{op}$ is homotopy terminal.
\end{proof}
\begin{rem}
Corollary \ref{cor:hocolim_restrict} also holds for ordinary colimits because homotopy initial implies initial.
%$\Phi$ being homotopy initial implies that $\Phi$ is initial.  
Indeed, the first condition requires contractible comma categories, while the second condition only requires connected comma categories.
\end{rem}

\begin{prop}
For all $\Gamma \in \GI$, the comparison map
$$
\alpha_\Gamma \colon [\mathbb{L}(U^{op})_! \Pi](\Gamma) \to [(U^{op})_! \Pi](\Gamma)
$$
is a weak equivalence.
\end{prop}
\begin{proof}
By Corollary \ref{cor:hocolim_restrict}, it is the same to show that the comparison map
$$
\underset{\mathcal{P}(\Gamma)^{op}}{\hocolim} \;\; (\Phi^{op})\hspace{-1pt}^*F \longrightarrow \underset{\mathcal{P}(\Gamma)^{op}}{\colim} \;\; (\Phi^{op})\hspace{-1pt}^*F
$$
is a weak equivalence, where $F$ is the composite
$$
(U^{op}/\Gamma) \longrightarrow \GI_2^{op} \overset{\Pi}{\longrightarrow} \Top_{G \times H}
$$
which is given on objects of $(U^{op}/\Gamma) = (\Gamma/U)^{op}$ by the formula
$$
F(\Gamma \to U(\Gamma', \Gamma'') \, ; \; (\Gamma', \Gamma'')) = \Conf(\Gamma', X) \times \Conf(\Gamma'', Y).
$$
The restriction $(\Phi^{op})\hspace{-1pt}^*F$, therefore, is given by
$$
(\Gamma', \Gamma'') \mapsto \Conf(\Gamma', X) \times \Conf(\Gamma'', Y).
$$
By Lemmas \ref{lem:intersections} and \ref{lem:cover}, this functor describes the poset of overlaps of an open cover of $\Conf(\Gamma, X \times Y)$.  
Since the union of the cover is the whole space,
$$
\underset{\mathcal{P}(\Gamma)^{op}}{\colim} \;\; (\Phi^{op})\hspace{-1pt}^*F \;\; \cong \;\;\Conf(\Gamma, X \times Y).
$$
We conclude that the map $\alpha_\Gamma$ is a weak equivalence by \cite[Corollary 3.3]{DI04}, a result of Dugger-Isaksen saying that a topological space is weak equivalent to the homotopy colimit of any open cover.
\end{proof}
%At this point, Theorem \ref{thm:pointwise} is proved.  

In the next section, we perform a similar analysis to accommodate morphisms from $\GI$ as well as objects.  %Due to the proliferation of slice categories, the notation becomes a bit dense.%; nevertheless, all the ideas are already present in the argument we have just given.

\subsection{Construction of the map $\beta$}
We give a formula for the components of $\beta$, and then verify that these components give a natural isomorphism.  Let $\psi \colon \Gamma \to U(\Gamma', \Gamma'')$ be a graph inclusion, and define a map
$$
B_{\psi} \colon \Conf(\Gamma', X) \times \Conf(\Gamma'', Y) \longrightarrow \Conf(\Gamma, X \times Y)
$$
by the formula
$$
((x_1, \ldots, x_m), (y_1, \ldots, y_m)) \mapsto ((x_{\psi(1)}, y_{\psi(1)}), \ldots, (x_{\psi(n)}, y_{\psi(n)}))
$$
where we have assumed for notational convenience that the vertices of $\Gamma$ are named $\{1, \ldots, n\}$ and the vertices of $\Gamma'$ and $\Gamma''$ are named $\{1, \ldots, m\}$.  This map manifestly lands in $(X \times Y)^n$, but its image actually lies in graphical configuration space.  Every edge $i \sim j$ of $\Gamma$ is sent to an edge $\psi(i) \sim \psi(j)$, and this edge must have been contributed by either $\Gamma'$ or $\Gamma''$, since it is present in the union.  Either way, the ordered pairs $(x_{\psi(i)}, y_{\psi(i)})$ and $(x_{\psi(j)}, y_{\psi(j)})$ will be distinct.  For example, if $\psi(i) \sim_{\Gamma'} \psi(j)$, then $x_{\psi(i)} \neq x_{\psi(j)}$, and so the ordered pairs differ in their first coordinate.  In the other case, the ordered pairs differ in their second coordinate.

For each graph $\Gamma$, write $B_{\Gamma}$ for the coproduct of the various maps $B_{\psi}$.
$$
B_{\Gamma} \colon \left[\coprod_{\substack{(\Gamma', \Gamma'') \in \GI_2 \\ \psi \colon \Gamma \to U(\Gamma', \Gamma'')}} \Conf(\Gamma', X) \times \Conf(\Gamma'', Y) \right] \; \xrightarrow{\coprod_{\psi} B_{\psi}} \Conf(\Gamma, X \times Y).
$$
Both the source and target of this map depend on $\Gamma$ contravariantly.  We introduce notation for the source considered as a functor of $\Gamma$:
$$
\begin{array}{rcccl}
P & \colon & \GI^{op} &\to& \Top_{G \times H} \\
&& \Gamma &\mapsto& \left[\coprod_{(\Gamma', \Gamma''), \psi} \Conf(\Gamma', X) \times \Conf(\Gamma'', Y) \right].
\end{array}
$$

%We prove several facts about the maps $B_{\Gamma} \colon P\Gamma \to \Conf(\Gamma, X \times Y)$, beginning with surjectivity.

We define a non-continuous function 
$$
\kappa_\Gamma \colon \Conf(\Gamma, X \times Y) \to P\Gamma
$$
and say in Proposition \ref{prop:section} that it is a section for the map $B_\Gamma$.  %Later, we will see that this map becomes continuous after modding out by an appropriate equivalence relation on $p\Gamma$.
To define the function, suppose that $q=((x_1, y_1), \ldots, (x_n, y_n)) \in \Conf(\Gamma, X \times Y)$.  Define two graphs $\Gamma'$ and $\Gamma''$ on the nodes $\{1,\ldots, n\}$ by the graphical relations
\begin{align*}
i \;\sim_{\Gamma'\phantom{'}} j \;\;\;\;\;&\Longleftrightarrow \;\;\;\;\; x_i \neq x_j \\
i \;\sim_{\Gamma''} j \;\;\;\;\;&\Longleftrightarrow \;\;\;\;\; y_i \neq y_j,
\end{align*}
and set
$$
\kappa_\Gamma(q) = ((x_1, \ldots, x_n), (y_1, \ldots, y_n))_{(\Gamma', \Gamma''),1_{\{1,\ldots,n\}}}
$$
where the subscript indicates the component of the coproduct.
\begin{prop} \label{prop:section}
For every $\Gamma$, we have $1_{\Conf(\Gamma, X \times Y)} = B_\Gamma \circ \kappa_\Gamma$.
\end{prop}
\begin{proof}
Immediate from the formulas for $B_\Gamma$ and $\kappa_\Gamma$.
\end{proof}

\begin{prop} \label{prop:commutes}
For every graph injection $\varphi \colon \Gamma_1 \to \Gamma_2$, we have 
$$
B_{\Gamma_1} \circ P\varphi = \Conf(\varphi, X \times Y) \circ B_{\Gamma_2}.
$$
\end{prop}
\begin{proof}
%It suffices to check the equality on every component of the disjoint union defining $P\Gamma$.  
Suppose $(\Gamma', \Gamma'') \in \GI_2$, and $\psi \colon \Gamma_2 \to U(\Gamma', \Gamma'')$ is a $\GI$-morphism.  Let 
$$
p = ((x_1, \ldots, x_m), (y_1, \ldots, y_m))_{(\Gamma', \Gamma''), \psi} \in P\Gamma%\Conf(\Gamma', X) \times \Conf(\Gamma'', Y)
$$
be a point, where the subscript indicates which component of the coproduct.  Now check
\begin{align*}
(B_{\Gamma_1} \circ P\varphi)(p) &= B_{\Gamma_1}((x_1, \ldots, x_m), (y_1, \ldots, y_m))_{(\Gamma', \Gamma''), \psi \circ \varphi}\\
&= B_{\psi \circ \varphi}((x_1, \ldots, x_m), (y_1, \ldots, y_m))\\
&= \left(\left(x_{(\psi \circ \varphi)(1)}, y_{(\psi \circ \varphi)(1)}\right), \ldots, \left(x_{(\psi \circ \varphi)(n)}, y_{(\psi \circ \varphi)(n)}\right)\right) \\
&= \left(\left(x_{\psi( \varphi(1))}, y_{\psi( \varphi(1))}\right), \ldots, \left(x_{\psi( \varphi(n))}, y_{\psi( \varphi(n))}\right)\right)\\
&= \Conf(\varphi, X \times Y)\left(\left(x_{\psi( 1)}, y_{\psi( 1)}\right), \ldots, \left(x_{\psi( m)}, y_{\psi( m)}\right)\right)\\
&= \left(\Conf(\varphi, X \times Y) \circ B_{\psi}\right)((x_1, \ldots, x_m), (y_1, \ldots, y_m))\\
&= \left(\Conf(\varphi, X \times Y) \circ B_{\Gamma_2}\right)(p).
\end{align*}
\end{proof}
\noindent
Recall that every $\GI_2$ morphism $\varphi \colon (\Gamma'_1, \Gamma''_1) \to (\Gamma'_2, \Gamma''_2)$ induces a continuous map
$$
\Pi \varphi \colon \Conf(\Gamma'_2, X) \times \Conf(\Gamma''_2, Y) \to \Conf(\Gamma'_1, X) \times \Conf(\Gamma''_1, Y).
$$

\begin{prop} \label{prop:descends}
For every graph $\Gamma \in \GI$ and graph injection $\psi \colon \Gamma \to U(\Gamma'_1, \Gamma''_1) $, we have
$$
(B_\psi) \circ (\Pi \varphi) = B_{(U\varphi) \circ \psi}.
$$
\end{prop}
\begin{proof}
For notational simplicity, suppose the vertex set of $\Gamma$ is $\{1, \ldots, n\}$, the vertex set of $\Gamma'_1$ and $\Gamma''_1$ is $\{1, \ldots, m\}$, and the vertex set of $\Gamma'_2$ and $\Gamma''_2$ is $\{1, \ldots, l\}$.

For a general point $((x_1, \ldots, x_l), (y_1, \ldots, y_l))$ of $\Conf(\Gamma'_2, X) \times \Conf(\Gamma''_2, Y)$,
\begin{align*}
((B_\psi) \circ (\Pi \varphi)) ((x_1, \ldots, x_l), (y_1, \ldots, y_l)) &= \left(B_\psi\right) \left(\left(x_{\varphi(1)}, \ldots, x_{\varphi(m)}\right), \left(y_{\varphi(1)}, \ldots, y_{\varphi(m)}\right)\right) \\
&= \left(\left(x_{\varphi(\psi(1))}, y_{\varphi(\psi(1))}\right), \ldots, \left(x_{\varphi(\psi(n))}, y_{\varphi(\psi(n))}\right)\right) \\
&= \left(\left(x_{(\varphi\circ\psi)(1)}, y_{(\varphi\circ\psi)(1)}\right), \ldots, \left(x_{(\varphi\circ\psi)(n)}, y_{(\varphi\circ\psi)(n)}\right)\right) \\
&= \left(B_{(U\varphi) \circ \psi}\right) ((x_{1}, \ldots, x_{l}), (y_{1}, \ldots, y_{l})),
\end{align*}
where we have written $\psi$ for the $\GI$ morphism as well as its underlying vertex injection, and similarly for $\varphi$.
\end{proof}

\begin{prop} \label{prop:standardize}
For every $p \in P\Gamma$, there is a $\GI_2$ morphism $\varphi_p \colon (\Gamma_1', \Gamma_1'') \to (\Gamma_2', \Gamma_2'')$ so
$$
(\Pi \varphi_p)(p) = (\kappa_\Gamma \circ B_\Gamma)(p).
$$
\end{prop}
\begin{proof}
Suppose
$$
p = ((x_1, \ldots, x_m), (y_1, \ldots, y_m))_{(\Gamma_2', \Gamma_2''), \psi},
$$
where the subscript indicates the component of the coproduct.  By the definition of $\kappa_\Gamma$, 
$$
(\kappa_\Gamma \circ B_\Gamma)(p) = ((x_{\psi(1)}, \ldots, x_{\psi(n)}), (y_{\psi(1)}, \ldots, y_{\psi(n)}))_{(\Gamma_1', \Gamma_1''), 1_{\{1, \ldots n\}}}
$$
where $\Gamma_1'$, and $\Gamma_1''$ contain as few edges as possible while still accommodating these two configurations.  It follows that we may set $\varphi_p = \psi$ to obtain the required equality.
\end{proof}
\noindent
By the colimit description of $[(U^{op})_!\Pi](\Gamma)$, this space is homeomorphic to the quotient
$$
P\Gamma / \hspace{-3pt} \sim \;\;\;\;\; \overset{\sim}{\longleftrightarrow} \;\;\;\;\; [(U^{op})_!\Pi](\Gamma)
$$
the equivalence relation $\sim$ is generated by identifications of the form
$$
p \sim (\Pi \varphi)(p)
$$
where $p \in P\Gamma$ is any point and $\varphi$ is any morphism of $U^{op} / \Gamma$.  By Proposition \ref{prop:descends}, the map $B_\Gamma$ factors uniquely through the quotient map
$$
P\Gamma \overset{Q_\Gamma}{\longrightarrow} [(U^{op})_!\Pi](\Gamma) \overset{\beta_\Gamma}{\longrightarrow} \Conf(\Gamma, X \times Y)
$$
defining a continuous map $\beta_\Gamma$ so that $B_\Gamma = \beta_\Gamma \circ Q_\Gamma$.
\begin{prop}
The morphisms $\beta_\Gamma$ form the components of a natural isomorphism.
\end{prop}
\begin{proof}
First, we argue that the morphisms $\beta_\Gamma$ form the components of a natural transformation.  Let $\varphi \colon \Gamma_1 \to \Gamma_2$ be a graph injection.  Since $Q_{\Gamma_1} \circ P\varphi = [(U^{op})_!\Pi](\varphi) \circ Q_{\Gamma_2}$, we have by Proposition \ref{prop:commutes}, 
\begin{align*}
B_{\Gamma_1} \circ P\varphi &= \Conf(\varphi, X \times Y) \circ B_{\Gamma_2} \\
\beta_{\Gamma_1} \circ Q_{\Gamma_1}  \circ P\varphi &= \Conf(\varphi, X \times Y) \circ \beta_{\Gamma_2} \circ Q_{\Gamma_2} \\
\beta_{\Gamma_1} \circ [(U^{op})_!\Pi](\varphi) \circ Q_{\Gamma_2} &= \Conf(\varphi, X \times Y) \circ \beta_{\Gamma_2} \circ Q_{\Gamma_2} \\
\beta_{\Gamma_1} \circ [(U^{op})_!\Pi](\varphi) &= \Conf(\varphi, X \times Y) \circ \beta_{\Gamma_2}
\end{align*}
where the last cancellation comes from the universal property of the quotient map $Q_{\Gamma_2}$.

We show that $\beta$ is an isomorphism by producing inverses for its components.  Recall from Proposition \ref{prop:section} that
$$
\kappa \colon \Conf(\Gamma, X \times Y) \to P\Gamma
$$ 
provides a discontinuous section to the continuous function $B_\Gamma$.  However, by two applications of Proposition~\ref{prop:standardize}, if any two points $p, p' \in P\Gamma$ have $B_\Gamma(p) = B_\Gamma(p')$, then
$$
p \sim (\kappa_\Gamma \circ B_\Gamma)(p) = (\kappa_\Gamma \circ B_\Gamma)(p') \sim p',
$$
and so these points are identified in the quotient.  In other words, any two points in the same fiber of $B_\Gamma$ are identified in $[(U^{op})_!\Pi](\Gamma)$, and so the composite function
$$
\Conf(\Gamma, X \times Y) \overset{\kappa_\Gamma}{\longrightarrow} P\Gamma \overset{Q_\Gamma}{\longrightarrow} [(U^{op})_!\Pi](\Gamma)
$$ 
is continuous since $\kappa$ is a section.  Moreover,
\begin{align*}
Q_\Gamma &=  Q_\Gamma \circ \kappa_\Gamma \circ B_\Gamma \\
Q_\Gamma &= Q_\Gamma \circ \kappa_\Gamma \circ \beta_\Gamma \circ Q_\Gamma \\
1 &= (Q_\Gamma \circ \kappa_\Gamma) \circ \beta_\Gamma
\end{align*}
since $Q_\Gamma$ is epic, and directly
\begin{align*}
\beta_\Gamma \circ (Q_\Gamma \circ \kappa_\Gamma) &= (\beta_\Gamma \circ Q_\Gamma) \circ \kappa_\Gamma \\
&= B_\Gamma \circ \kappa_\Gamma \\
&= 1,
\end{align*}
and so the composite $(Q_\Gamma \circ \kappa_\Gamma)$ is a continuous two-sided inverse for $\beta_\Gamma$.
\end{proof}
\subsection{Proof of Theorem \ref{thm:pointwise}} \label{sec:pointwise_proof}
By Theorem \ref{thm:main}, we have a homotopy equivalence
$$
 \underset{U^{op}/\Gamma}{\hocolim} \;\; \left[(U^{op}/\Gamma) \longrightarrow \GI_2^{op} \overset{\Pi}{\longrightarrow} \Top_{G \times H} \right] \simeq \Conf(\Gamma, X \times Y),
$$
and moreover, this equivalence is functorial in the variable $\Gamma \in \GI$.  By first applying Proposition \ref{prop:initial}, and then Proposition \ref{prop:initiala}, we may replace the indexing category for the homotopy colimit with the poset $\mathcal{P}(\Gamma)$, and then the slice category $(U^{op}_n/\Gamma)$, obtaining
$$
 \underset{U^{op}_n/\Gamma}{\hocolim} \;\; \left[(U^{op}_n/\Gamma) \longrightarrow \mathcal{G}(n) \times \mathcal{G}(n) \xrightarrow{\Pi|_{\mathcal{G}(n) \times \mathcal{G}(n)}}\Top_{G \times H} \right] \simeq \Conf(\Gamma, X \times Y).
$$
The left-hand-side of this expression gives a formula for $\mathbb{L}(U^{op}_n)_! \Pi$, so we are done.
% Compare to the argument about opfibrations Prop. 4.1 from https://ncatlab.org/nlab/show/Beck-Chevalley+condition#PullbacksOfOpfibrations
\bibliographystyle{amsalpha}
\bibliography{math}

\newcommand{\etalchar}[1]{$^{#1}$}
\providecommand{\bysame}{\leavevmode\hbox to3em{\hrulefill}\thinspace}
\providecommand{\MR}{\relax\ifhmode\unskip\space\fi MR }
\providecommand{\MRhref}[2]{%
  \href{http://www.ams.org/mathscinet-getitem?mr=#1}{#2}
}
\providecommand{\href}[2]{#2}
\begin{thebibliography}{DGJRS18}

\bibitem[ADCK18]{An:2018aa}
Byung~Hee An, Gabriel~C. Drummond-Cole, and Ben Knudsen, \emph{Edge
  stabilization in the homology of graph braid groups}.

\bibitem[Arn69]{Arnold69}
V.I. Arnol'd, \emph{The cohomology ring of the colored braid group},
  Mathematical notes of the Academy of Sciences of the USSR \textbf{5} (1969),
  no.~2, 138--140 (English).

\bibitem[Bib16]{Bibby16}
Christin Bibby, \emph{Cohomology of abelian arrangements}, Proc. Amer. Math.
  Soc. \textbf{144} (2016), no.~7, 3093--3104. \MR{3487239}

\bibitem[Bri73]{Brieskorn73}
Egbert Brieskorn, \emph{Sur les groupes de tresses [d'apr\`es {V}. {I}.
  {A}rnol' d]}, 21--44. Lecture Notes in Math., Vol. 317. \MR{0422674}

\bibitem[CEF15]{ChurchEllenbergFarb15}
T.~Church, J.~Ellenberg, and B.~Farb, \emph{F{I}-modules and stability for
  representations of symmetric groups}, Duke Math. J. \textbf{164} (2015),
  no.~9, 1833--1910.

\bibitem[Coh73]{CohenThesis}
F.~Cohen, \emph{Cohomology of braid spaces}, Bull. Amer. Math. Soc. \textbf{79}
  (1973), 763--766.

\bibitem[CT78]{CT78}
F.~R. Cohen and L.~R. Taylor, \emph{Computations of {G}el' fand-{F}uks
  cohomology, the cohomology of function spaces, and the cohomology of
  configuration spaces}, Geometric applications of homotopy theory ({P}roc.
  {C}onf., {E}vanston, {I}ll., 1977), {I}, Lecture Notes in Math., vol. 657,
  Springer, Berlin, 1978, pp.~106--143. \MR{513543}

\bibitem[dBW17]{dBW17}
Pedro~Boavida de~Brito and Michael~S. Weiss, \emph{The configuration category
  of a product}, 2017.

\bibitem[DGJRS18]{MFO18}
G.~Denham, G.~Gaiffi, R.~Jim{\'e}nez~Rolland, and A.~Suciu, \emph{Topology of
  arrangements and representation stability}, MFO reports \textbf{2} (2018),
  38--39.

\bibitem[DHK17]{DHK17}
William Dwyer, Kathryn Hess, and Ben Knudsen, \emph{Configuration spaces of
  products}, 2017, To appear in Transactions of the AMS.

\bibitem[DI04]{DI04}
Daniel Dugger and Daniel~C. Isaksen, \emph{Topological hypercovers and {$\Bbb
  A^1$}-realizations}, Math. Z. \textbf{246} (2004), no.~4, 667--689.
  \MR{2045835}

\bibitem[dLS01]{dLS01}
Mark de~Longueville and Carsten~A. Schultz, \emph{The cohomology rings of
  complements of subspace arrangements}, Math. Ann. \textbf{319} (2001), no.~4,
  625--646. \MR{1825401}

\bibitem[Dup15]{Dupont15}
Cl\'ement Dupont, \emph{The {O}rlik-{S}olomon model for hypersurface
  arrangements}, Ann. Inst. Fourier (Grenoble) \textbf{65} (2015), no.~6,
  2507--2545. \MR{3449588}

\bibitem[FZ00]{FZ00}
Eva~Maria Feichtner and G\"unter~M. Ziegler, \emph{The integral cohomology
  algebras of ordered configuration spaces of spheres}, Doc. Math. \textbf{5}
  (2000), 115--139. \MR{1752611}

\bibitem[GGSX15]{GGX15}
Jes\'us Gonz\'alez, Aldo Guzm\'an-S\'aenz, and Miguel Xicot\'encatl, \emph{The
  cohomology ring away from 2 of configuration spaces on real projective
  spaces}, Topology Appl. \textbf{194} (2015), 317--348. \MR{3404620}

\bibitem[GS]{M2}
Daniel~R. Grayson and Michael~E. Stillman, \emph{Macaulay2, a software system
  for research in algebraic geometry}, Available at
  \texttt{http://www.math.uiuc.edu/Macaulay2/}.

\bibitem[Kri94]{Kriz94}
Igor Kriz, \emph{On the rational homotopy type of configuration spaces}, Ann.
  of Math. (2) \textbf{139} (1994), no.~2, 227--237. \MR{1274092}

\bibitem[Mag16]{Maguire16}
Megan Maguire, \emph{Computing cohomology of configuration spaces}, 2016.

\bibitem[MPY17]{MPY17}
Daniel Moseley, Nicholas Proudfoot, and Ben Young, \emph{The {O}rlik-{T}erao
  algebra and the cohomology of configuration space}, Exp. Math. \textbf{26}
  (2017), no.~3, 373--380. \MR{3642114}

\bibitem[NS17]{NS17}
Rohit Nagpal and Andrew Snowden, \emph{Periodicity in the cohomology of
  symmetric groups via divided powers}, 2017.

\bibitem[OS80]{OS80}
Peter Orlik and Louis Solomon, \emph{Combinatorics and topology of complements
  of hyperplanes}, Invent. Math. \textbf{56} (1980), no.~2, 167--189.
  \MR{558866}

\bibitem[Ram18]{Ramos2018}
Eric Ramos, \emph{Stability phenomena in the homology of tree braid groups},
  Algebraic {\&} Geometric Topology \textbf{18} (2018), no.~4, 2305--2337.

\bibitem[S{\etalchar{+}}17]{sage}
W.\thinspace{}A. Stein et~al., \emph{{S}age {M}athematics {S}oftware ({V}ersion
  7.5.1)}, The Sage Development Team, 2017, {\tt http://www.sagemath.org}.

\bibitem[Tos16]{Tosteson16}
P.~Tosteson, \emph{Lattice spectral sequences and cohomology of configuration
  spaces}, preprint, 2016.

\bibitem[Tot96]{Totaro96}
Burt Totaro, \emph{Configuration spaces of algebraic varieties}, Topology
  \textbf{35} (1996), no.~4, 1057--1067. \MR{1404924}

\bibitem[WG17]{WG17}
John~D. Wiltshire-Gordon, \emph{Models for configuration space in a simplicial
  complex}, 2017.

\end{thebibliography}

\end{document}